\documentclass[11pt,a4paper]{article}
\usepackage[T1]{fontenc}
\usepackage{url}
\usepackage{amssymb}
\usepackage{amsmath}
\usepackage{mathrsfs}
\usepackage[english]{babel}

\oddsidemargin=\evensidemargin
 
\newtheorem{theorem}{Theorem}[section]

\newtheorem{lemma}[theorem]{Lemma}

\newtheorem{definition}[theorem]{Definition}
\newtheorem{remark}[theorem]{Remark}
\numberwithin{equation}{section}

\newcommand{\Rot}{\operatorname{\mathbf{curl}}}
\newcommand{\Div}{\operatorname{\mathrm{div}}}

\newcommand{\rot}{\operatorname{\mathrm{curl}}}

\newcommand{\loc}{\mathrm{loc}}
\renewcommand{\Im}{\operatorname{Im}}
\renewcommand{\Re}{\operatorname{Re}}
 
\newcommand  {\C}{{\mathbb C}}

\newcommand  {\Q}{{\mathbb Q}} 
\newcommand  {\R}{{\mathbb R}}

\newcommand {\Id}{\mathrm {I}}

\renewcommand{\SS}{\boldsymbol{\mathsf S}}
\newcommand  {\TT}{\boldsymbol{\mathsf T}}

\newcommand  {\LL}{\boldsymbol{\mathsf L}}
\newcommand  {\HH}{\boldsymbol{\mathsf H}}
\newcommand  {\EE}{\boldsymbol{\mathsf E}}

\newcommand  {\D}{\mathsf D}

\newcommand  {\nn}{\boldsymbol{\mathsf n}}

\newcommand  {\uu}{\boldsymbol{\mathsf u}}
\newcommand  {\jj}{\mathop{\boldsymbol{\mathsf j}}\nolimits}
\newcommand  {\mm}{\boldsymbol{\mathsf m}}
\newcommand  {\vv}{\boldsymbol{\mathsf v}}

\newenvironment{proof}{\begin{trivlist}
                       \item[]{\sc Proof. }}{\hfill $\blacksquare$
                     \end{trivlist}} 
\begin{document}
\title{On the Kleinman-Martin integral equation method for electromagnetic scattering by a dielectric body}
\author{
  Martin Costabel
  \thanks{IRMAR, Institut Math\'ematique, Universit\'e de Rennes 1, 35042
    Rennes, France, martin.costabel@univ-rennes1.fr}
  \and
  Fr\'ed\'erique Le Lou\"er
  \thanks{IRMAR, Institut Math\'ematique, Universit\'e de Rennes 1, 35042
    Rennes}
  }
\date{}
\maketitle

\begin{abstract}
 The interface problem describing the scattering of time-harmonic electromagnetic waves by a dielectric body is often formulated as a pair of coupled boundary integral equations for the electric and magnetic current densities on the interface $\Gamma$.  In this paper, following an idea developed by Kleinman and Martin \cite{KlMa} for acoustic scattering problems, we consider methods for solving the dielectric scattering problem using a single integral equation over $\Gamma$ for a single unknown density. One knows that such boundary integral formulations of the Maxwell equations are not uniquely solvable when the exterior wave number is an eigenvalue of an associated interior Maxwell boundary value problem. We obtain four different families of integral equations for which we can show that by choosing some parameters in an appropriate way, they become uniquely solvable for all real frequencies. We analyze the well-posedness of the integral equations in the space of finite energy on smooth and non-smooth boundaries. 
\end{abstract}


\section{Introduction} 

We consider the scattering of time-harmonic electromagnetic waves in $\R^3$ by a bounded Lipschitz obstacle. We assume that the dielectric permittivity and the magnetic permeability take constant, in general different, values in the interior and in the exterior of the domain. This problem is described by the system of Maxwell's equations, valid in the sense of distributions in $\R^3$, which implies two transmission conditions expressing the continuity of the tangential components of the fields across the interface. The transmission problem is completed by the Silver-M\"uller radiation condition at infinity (see \cite{Monk} and \cite{Ned}).

 It is well known that this problem can be reduced in several different ways to systems of two boundary integral equations for two unknown tangential vector fields on the interface. Such formulations are analyzed in Harrington's book \cite{Harrington} and  in Martin and  Ola's comprehensive paper \cite{MartinOla}. Some pairs of boundary integral equations, such as M\"uller's \cite{Muller}, are uniquely solvable for all real values of the exterior wave number and  others, such as the so-called electric-field formulations \cite{MartinOla} are not, although the underlying Maxwell interface problem is always uniquely solvable under standard assumptions on the material coefficients.
 
 More recent research works in the scientific and engineering community show that there are computational advantages to solve dielectric scattering  problems via a single integral equation for a single unknown, rather than a system of two equations of two unknowns. For two-dimensional dielectric scattering problems, one can find various formulations and numerical results in \cite{Pomp,Seydou,SwatekCiric}. In \cite{Marx} Marx  develops  single source  integral formulations for three-dimensional homogeneous dielectric objects using an ansatz on the exterior electric field and  in \cite{Yeung} Yeung presents electric-field (EFIE) and magnetic-field (MFIE) integral equations for a single unknown based on an ansatz on the interior scattered field.  Computational results in \cite{Yeung} show higher convergence speed for the MFIE and the EFIE than for the pairs of boundary integral equations. However, both of these single integral formulations suffer from spurious non-unique solvability due to interior resonances. 

 In this paper, we study methods for solving the transmission problem using a single boundary integral equation for a single unknown tangential vector field on the interface by eliminating irregular frequencies. We follow ideas of \cite{KlMa} where Kleinman and  Martin considered the analogous question for the acoustic interface scattering problem. The method consists of representing the solution in one domain by some combination of a single layer potential and a double layer potential, and inserting this representation into the transmission conditions and the Calder\'on relations of the other domain. Several different integral equations of the first kind or of the second kind, containing two arbitrary parameters, can be obtained in this way, and in the scalar case, the parameters can be chosen in such a way that no spurious real frequencies are introduced.
Following the same procedure in the electromagnetic case, one encounters two main difficulties:

 The first problem is that some boundary integral operators that are compact in the scalar case are no longer compact, and therefore arguments based on the theory of Fredholm integral equations of the second kind have to be refined in order to show well-posedness of the corresponding integral equations.

 The second problem comes from a lack of ellipticity. The spurious frequencies are associated with the spectrum of a certain interior boundary value problem of the third kind, and whereas in the scalar case this is an elliptic boundary value problem whose spectrum can be moved off the real line by the right choice of parameters, in the Maxwell case this boundary value problem is not elliptic, in general. Thus an additional idea is needed to avoid real irregular frequencies. Mautz presents in \cite{Mautz} an alternative that leads to an  associated interior problem with an impedance boundary condition, which ensures the uniqueness of the solution. However Mautz's equation is not adapted to our point of view of variational methods and energy spaces.  Since the Kleinman-Martin method has similarities to the combined field integral equation method, we use a regularizer introduced by Steinbach and Windisch in \cite{Wind} in the context of combined field integral equations for the time-harmonic Maxwell equations. This regularizer is a positive definite boundary integral operator with a similar structure as the operator of the electrical field integral equation, but it is not a compact operator like those used in \cite{Colton} and \cite{BuHi} for regularizing the exterior electromagnetic scattering problem. Its introduction changes the boundary condition in the associated interior boundary value problem from a non-elliptic local impedance-like condition to a non-local, but elliptic, boundary condition.  

 This work contains results from the thesis \cite{FLL} where this integral formulation of the transmission problem is used to study the shape derivatives of the solution of the dielectric scattering problem, in the context of a problem of optimizing the shape of a dielectric lens in order to obtain a prescribed radiation pattern. 

 In Section~\ref{TrPot} we recall some results about traces and potentials for Maxwell's equations in Sobolev spaces. We use the notation of \cite{BuHiPeSc} and \cite{BuHi} and quote some important properties of the boundary integral operators that constitute the Calder\'on projector for Maxwell's equations.
 
Sections~\ref{IE1} and \ref{IE2} contain the details of the method for solving the transmission problem using single-source boundary integral equations. In Section~\ref{IE1}, we start from a layer representation for the exterior field whereas in Section~\ref{IE2}, we use a layer representation for the interior field. In either case, we derive two boundary integral equations of the second kind and we show uniqueness of their solutions under suitable conditions on an associated interior boundary value problem. Moreover, we show that the integral operators in each integral equation are Fredholm of index zero. We also construct the solution of the transmission problem using the solution of any of the four integral equations. We finally show how to choose the free parameters so that the associated interior boundary value problem is uniquely solvable, and as a consequence, we can construct an integral representation of the solution which yields uniquely solvable boundary integral equations for all real frequencies. 

For smooth domains, we base the analysis of the integral operators on the technique of Helmholtz decomposition, which represents a tangential vector field by two scalar field and each integral operator acting on tangential fields by a two-by-two matrix of scalar operators. Since these operators then act between standard Sobolev spaces instead of the complicated mixed-order energy space, it is easy to check for compactness or ellipticity. Using this technique, we find rather general sufficient conditions on the physical parameters to ensure unique solvability of the integral equations.
If the boundary is only Lipschitz, we show that under more restrictive conditions one still has strong ellipticity of the integral operators. This conditions include the physically relevant case of positive permeabilities, permittivities, and frequencies.

 
\section{The dielectric scattering problem}

Let $\Omega$ denote a bounded domain in $\R^{3}$  and let $\Omega^c$ denote the exterior domain $\R^3\backslash\overline{\Omega}$. 
In this paper, we will assume that the boundary $\Gamma$ of $\Omega$ is a Lipschitz continuous and simply connected closed surface. Let $\nn$ denote the outer unit normal vector on the boundary $\Gamma$.

In $\Omega$ (resp. $\Omega^c$) the electric permittivity $\epsilon_{i}$ (resp.
$\epsilon_{e}$) and the magnetic permeability $\mu_{i}$ (resp. $\mu_{e}$) are positive constants. The frequency $\omega$ is the same in $\Omega$ and in
$\Omega^c$. The interior wave number $\kappa_{i}$ and the exterior wave   number $\kappa_{e}$ are complex constants of non negative imaginary part.

\textbf{Notation: }
For a domain $G\subset\R^3$ we denote by $H^s(G)$ the usual $L^2$-based Sobolev space of order $s\in\R$, and by $H^s_{\loc}(\overline G)$ the space of functions whose restrictions to any bounded subdomain $B$ of $G$ belong to $H^s(B)$, with the convention $H^0\equiv L^2$. Spaces of vector functions will be denoted by boldface letters, thus 
$$
 \HH^s(G)=(H^s(G))^3\,.
$$ 
If $\D$ is a differential operator, we write:
\begin{eqnarray*}
\HH(\D,\Omega)& = &\{ u \in \LL^2(\Omega) : \D u \in \LL^2(\Omega)\}\\
\HH_{\loc}(\D,\overline{\Omega^c})&=& \{ u \in \LL_{\loc}^2(\overline{\Omega^c}) : \D u \in
\LL_{\loc}^2(\overline{\Omega^c}) \}\end{eqnarray*}
The space $\HH(\D, \Omega)$ is endowed with the natural graph norm. This defines in particular the Hilbert spaces $\HH(\Rot,\Omega)$ and $\HH(\Rot\Rot,\Omega)$.

\bigskip

The time-harmonic dielectric scattering problem is formulated as follows.

\medskip
 
\textbf{The dielectric scattering problem :} \\
Given an incident field 
$\EE^{inc}\in\HH_{\loc}(\Rot,\R^3)$ that satisfies 
$\Rot\Rot \EE^{inc} - \kappa_{e}^2\EE^{inc} =0$
in a neighborhood of $\overline{\Omega}$,
we seek  two fields $\EE^{i}\in \HH(\Rot,\Omega)$ and $\EE^{s}\in
\HH_{\loc}(\Rot,\overline{\Omega^c})$ satisfying the time-harmonic Maxwell equations
\begin{eqnarray}
\label{(1.2a)}
  \Rot\Rot \EE^{i} - \kappa_{i}^2\EE^{i}& = 0&\text{ in }\Omega,
\\
\label{(1.2b)} 
  \Rot\Rot \EE^{s} - \kappa_{e}^2\EE^{s} &= 0&\text{ in }\Omega^c,\end{eqnarray}
the two transmission conditions, 
\begin{eqnarray}{}
\label{T1}&\,\,\;\;\nn\times \EE^{i}=\nn\times( \EE^{s}+\EE^{inc})&\qquad\text{ on }\Gamma
\\
\label{T2} &\mu_{i}^{-1}(\nn\times\Rot \EE^{i}) = \mu_{e}^{-1}\nn\times\Rot(\EE^{s}+\EE^{inc})&\qquad\text{ on }\Gamma\end{eqnarray} 
and the Silver-M\"uller radiation condition:
\begin{equation}\label{T3}\lim_{|x|\rightarrow+\infty}|x|\left| \Rot \EE^{s}(x)\times\frac{x}{| x |}- i\kappa_{e}\EE^{s}(x) \right|
 =0.\end{equation}

It is well known that this problem has at most one solution under some mild restrictions on the dielectric constants. We give sufficient conditions in the next theorem, and for completeness we give its simple proof.
\begin{theorem}
\label{t1} 
 Assume that the constants $\mu_{i}$, $\kappa_{i}$, $\mu_{e}$ and $\kappa_{e}$ satisfy:
\begin{enumerate}
\item[(i)]
 $\kappa_{e}$ is real and positive or $\Im(\kappa_{e})>0$,\\[-2ex]
 \item[(ii)] 
 $\Im\left(\overline{\kappa_{e}}\dfrac{\mu_{e}}{\mu_{i}}\right)\leq0$ 
 and $\Im\left(\overline{\kappa_{e}}\dfrac{\mu_{e}}{\mu_{i}}\kappa_{i}^2\right)\ge0$.
\end{enumerate}
Then the dielectric scattering problem has at most one solution.
 \end{theorem}
 
\begin{proof} 
 We use similar arguments as in the acoustic case \cite{KlMa}. 
 Assume that $\EE^{inc} = 0$. Let $(\EE^{i},\EE^{s})$ be a solution of the homogeneous scattering problem. Let $B_{R}$ be a ball of radius $R$ large enough such that ${\Omega\subset B_{R}}$ and let $\nn_{R}$ the unit outer normal vector to $B_{R}$. Integration by parts using the Maxwell equations \eqref{(1.2a)} and \eqref{(1.2b)} and the transmission conditions \eqref{T1} and \eqref{T2}
 gives :
$$
\int_{\partial{B_{R}}}(\Rot\EE^{s}\times \nn_{R})\cdot \overline{\EE^{s}}
=   \int\limits_{B_{R}\backslash\Omega}\{|\Rot \EE^{s}|^{2}-\kappa_{e}^{2}|
\EE^{s}|^{2}\}  
+ \dfrac{\mu_{e}}{\mu_{i}}\int_{\Omega}\{|\Rot \EE^{i}|^2
-\kappa_{i}^2| \EE^{i}|^2\}
$$
We multiply this by $\overline{\kappa_{e}}$  and take the imaginary part:
\[ 
\begin{split}
 \Im\left(\overline{\kappa_{e}}\displaystyle{\int_{\partial{B_{R}}}(\Rot\EE^s\times \nn_{R})\cdot \overline{\EE^s}}\right)
&=   \Im(\overline{\kappa_{e}})\left(\int_{B_{R}\setminus\Omega}
\{|\Rot \EE^s|^{2}+|\kappa_{e} \EE^{s}|^{2}\}\right) \\
 + &
 \Im\left(\overline{\kappa_{e}}\dfrac{\mu_{e}}{\mu_{i}}\right)\int_{\Omega}|\Rot \EE^{i}|^2
-\Im\left(\overline{\kappa_{e}}\dfrac{\mu_{e}}{\mu_{i}}\kappa_{i}^2\right)\displaystyle{\int_{\Omega}| \EE^{i}|^2}.
\end{split}
\]
Under the hypotheses (i) and (ii), all terms on the right hand side are non-positive.

Thanks to the Silver-M\"uller condition, we have
$$\begin{array}{l}\lim\limits_{R\rightarrow+\infty}\displaystyle{\int_{\partial B_{R}}|\Rot\EE^s\times\nn_{R}-i\kappa_{e}\EE^s|^2=0.}\end{array}
$$ 
Developing this expression, we get
$$\lim\limits_{R\rightarrow+\infty}\int_{\partial B_{R}}|\Rot\EE^s\times\nn_{R}|^2+|\kappa_{e} \EE^s|^2-2\Re\left(\Rot\EE^s\times\nn_{R}\cdot\overline{i\kappa_{e}\EE^s}\right)=0.
$$
As we have seen, we have 
$$
 \int_{\partial B_{R}}\Re\left(\Rot\EE^s\times\nn_{R}\cdot\overline{i\kappa_{e}\EE}\right)=\Im\int_{\partial B_{R}}\left(\overline{\kappa_{e}}\Rot\EE^s\times\nn_{R}\cdot\overline{\EE}\right)\leq0.
$$ 
It follows that 
$$
 \lim_{R\rightarrow+\infty}\int_{\partial{B_{R}}}|\EE^{s}|^2 = 0.
$$
Thus, by Rellich's lemma \cite{Colton}, $\EE^{s}= 0$ in $\Omega^c$. Using the transmission conditions, we obtain $\gamma_{D}\EE^{i}=\gamma_{N_{\kappa_{i}}}\EE^i=0$. It follows that  $\EE^{i}=0$ in $\Omega$.
\end{proof}

\section{Traces and electromagnetic potentials} \label{TrPot}

We use some well known results about traces of vector fields and integral representations of time-harmonic electromagnetic fields on a bounded domain $\Omega$. Details can be found 
in \cite{BuCi2, BuMCSc, BuCosShe,BuHiPeSc, MartC, Ned}. Recall that the boundary $\Gamma$ is only assumed to be Lipschitz continuous, unless stated otherwise.

\begin{definition}\label{2.1} For  a vector function $\uu\in (\mathscr{C}^{\infty}(\overline{\Omega}))^3$ and a scalar function $v\in\mathscr{C}^{\infty}(\overline{\Omega})$ we define the
traces :
\begin{align*}
\gamma v&=v_{|_{\Gamma}}\\
\gamma_{D}\uu&=(\nn\times \uu)_{|_{\Gamma}}\textrm{ (Dirichlet)}\\
\gamma_{N_{\kappa}}\uu&=\kappa^{-1}(\nn\times\Rot \uu)_{|_{\Gamma}}\textrm{ (Neumann).}
\end{align*}
\end{definition}

We use standard Sobolev spaces  $H^t(\Gamma)$, $t\in[-1,1]$, endowed with standard norms $||\cdot||_{H^t(\Gamma)}$ and with the  convention $H^0(\Gamma)=~L^2(\Gamma)$. Spaces of vector densities are denoted by boldface letters, thus $\HH^t(\Gamma)=\big(H^t(\Gamma)\big)^3$. We define spaces of tangential vector fields as  $\HH_{\times}^{t}(\Gamma)=\nn\times\HH^{t}(\Gamma)$.
Note that on non-smooth boundaries, the latter space different, in general, from the other space of  tangential vector fields 
$\HH_{\|}^{t}(\Gamma)=\nn\times\HH_{\times}^{t}(\Gamma)$. On smooth boundaries, the two spaces coincide. 
For $s=0$ we set $\HH^0_{\times}(\Gamma)=\LL_{\times}^2(\Gamma)$.

The trace maps 
\begin{align*}
\gamma:H^{s+\frac{1}{2}}(\Omega)&\rightarrow H^{s}(\Gamma), \\
\gamma_{D}:\HH^{s+\frac{1}{2}}(\Omega)&\rightarrow \HH_{\times}^{s}(\Gamma)
\end{align*}
 are continuous for all $s>0$, if the domain is smooth. On a polyhedron, the trace maps are continuous for $s\in(0,2)$, whereas for a general bounded Lipschitz domain in general, the validity is only given for $s\in(0,1)$.
For $s=1$, the trace operator $\gamma$ fails, in general, to map $H^{\frac{3}{2}}(\Omega)$ to $H^{1}(\Gamma)$, although $H^{1}(\Gamma)$ is well defined on the boundary $\Gamma$, see \cite{JerisonKenig95}.
 
The dual spaces of $H^t(\Gamma)$ and $\HH_{\times}^{t}(\Gamma)$ 
   with respect to the $L^2$ (or $\LL^2$) scalar product is denoted by $H^{-t}(\Gamma)$ and $\HH_{\times}^{-t}(\Gamma)$, respectively. 
 
We use the  surface differential operators: The tangential gradient denoted by $\nabla_{\Gamma}$, the surface divergence denoted by $\Div_{\Gamma}$, the tangential vector curl denoted by $\Rot_{\Gamma}$ and the surface scalar curl denoted by $\rot_{\Gamma}$. For their definitions we refer to \cite{BuCosShe}, \cite{MartC} and \cite{Ned}.
 
\begin{definition} We define the Hilbert space
$$
  \HH_{\times}^{-\frac{1}{2}}(\Div_{\Gamma},\Gamma)=\left\{ \jj\in
\HH_{\times}^{-\frac{1}{2}}(\Gamma),\Div_{\Gamma}\jj \in
H^{-\frac{1}{2}}(\Gamma)\right\}
$$
endowed with the norm 
$$
||\cdot||_{\HH_{\times}^{-\frac{1}{2}}(\Div_{\Gamma},\Gamma)}=||\cdot||_{\HH^{-\frac{1}{2}}_{\times}(\Gamma)}+||\Div_{\Gamma}\cdot||_{H^{-\frac{1}{2}}(\Gamma)}.
$$ 
\end{definition}

The skew-symmetric bilinear form 
$$
\begin{array}{rccl}\mathbf{\mathcal{B}}:&\HH_{\times}^{-\frac{1}{2}}(\Div_{\Gamma},\Gamma)\times\HH_{\times}^{-\frac{1}{2}}(\Div_{\Gamma},\Gamma)&\rightarrow&\C\vspace{.1cm}\\&(\;\jj,\mm)&\rightarrow&\mathbf{\mathcal{B}}(\;\jj,\;\mm)=\displaystyle{\int_{\Gamma}\jj\cdot(\mm\times\nn)}\;d\sigma\end{array}
$$ 
defines a non-degenerate duality product on 
$\HH_{\times}^{-\frac{1}{2}}(\Div_{\Gamma},\Gamma)$.

\begin{lemma}
\label{2.3} The operators $\gamma_{D}$ and $\gamma_{N}$ are linear and
continuous from $(\mathscr{C}^{\infty}(\overline{\Omega}))^3$ to $\LL_{\times}^2(\Gamma)$ and they
can be extended to continuous linear operators from $\HH(\Rot,\Omega)$ and 
$\HH(\Rot,\Omega)\cap\HH(\Rot\Rot,\Omega)$, respectively, to
$\HH_{\times}^{-\frac{1}{2}}(\Div_{\Gamma},\Gamma)$. 
Moreover, for all $\uu,\;\vv\in\HH(\Rot,\Omega)$, we have:
\begin{equation}
\label{IPP2}
\int_{\Omega}\left[\left(\Rot\uu\cdot\vv\right)-\left(\uu\cdot\Rot\vv\right)\right]dx=\mathcal{B}(\gamma_{D}\vv,\gamma_{D}\uu).
\end{equation}
\end{lemma}

For  $\uu\in \HH_{\loc}(\Rot,\overline{\Omega^c})$ and $\vv \in
\HH_{\loc}(\Rot\Rot,\overline{\Omega^c}))$
we define $\gamma_{D}^c\uu $ and $\gamma_{N}^c\vv $ in the same way and the same mapping properties hold true.\medskip

Let $\kappa$ be a complex number such that $\Im(\kappa)\ge0$ and let 
$$
  G(\kappa,|x-y|)=\dfrac{e^{i\kappa|x-y|}}{4\pi| x-y|}
$$
be the fundamental solution of the Helmholtz equation 
$$
  {\Delta u + \kappa^2u =0}.
$$
The single layer potential  $\psi_{\kappa}$ is given by   :
\begin{center}$(\psi_{\kappa}u)(x) = \displaystyle{\int_{\Gamma}G(\kappa,|x-y|)u(y) d\sigma(y)}\qquad x
\in\R^3\backslash\Gamma$,\end{center}
and its trace by 
$$
 V_{\kappa}u(x)= \int_{\Gamma}G(\kappa,|x-y|)u(y) d\sigma(y)\qquad x
\in\Gamma.
$$
For a proof of the following well-known result, see \cite{HsiaoWendland,Ned}.
\begin{lemma}
\label{3.1}  
The operators
$$ 
\begin{array}{ll} \psi_{\kappa} & : H^{-\frac{1}{2}}(\Gamma)\rightarrow H^1_{\loc}(\R^3) \vspace{2mm} \\ 
V_{\kappa}& : H^{-\frac{1}{2}}(\Gamma)\rightarrow H^{\frac{1}{2}}(\Gamma) \end{array}
$$  
are continuous.
\end{lemma}
We define the electric potential $\Psi_{E_{\kappa}}$ generated by $\jj\in
\HH_{\times}^{-\frac{1}{2}}(\Div_{\Gamma},\Gamma) $ by 
$$\Psi_{E_{\kappa}}\jj :=
\kappa\psi_{\kappa}\jj + \kappa^{-1}\nabla\psi_{\kappa}\Div_{\Gamma}\jj
$$
This can be written as $\Psi_{E_{\kappa}}\jj :=
\kappa^{-1}\Rot\Rot\psi_{\kappa}\jj$ because of the Helmholtz equation and
the identity $\Rot\Rot = -\Delta +\nabla\Div$ (cf. \cite{BuCi2}).

We define the magnetic potential  
$\Psi_{M_{\kappa}}$ generated by $\mm\in \HH_{\times}^{-\frac{1}{2}}(\Div_{\Gamma},\Gamma) $
by 
$$
 \Psi_{M_{\kappa}}\mm := \Rot\psi_{\kappa}\mm.
$$
 These potentials satisfy 
$$
 \kappa^{-1}\Rot\Psi_{E_{\kappa}} =
\Psi_{M_{\kappa}}\quad\text{ and }\quad\kappa^{-1}\Rot\Psi_{M_{\kappa}} =
\Psi_{E_{\kappa}}.
$$
We denote the identity operator by $\Id$.
\begin{lemma}
\label{3.2}
The potential operators $\Psi_{E_{\kappa}}$ and $\Psi_{M_{\kappa}}$ are
continuous from $\HH_{\times}^{-\frac{1}{2}}(\Div_{\Gamma},\Gamma)$ to
$\HH_{\loc}(\Rot,\R^3)$. 
For $\jj\in \HH_{\times}^{-\frac{1}{2}}(\Div_{\Gamma},\Gamma)$ we have 
$$(\Rot\Rot -\kappa^2\Id)\Psi_{E_{\kappa}}\jj = 0\textrm{ and }(\Rot\Rot
-\kappa^2\Id)\Psi_{M_{\kappa}}\mm = 0\textrm{ in }\R^3\backslash\Gamma$$ and  $\Psi_{E_{\kappa}}\jj$
and $\Psi_{M_{\kappa}}\mm$ satisfy the Silver-M\"uller condition.
\end{lemma}
It follows that the traces $\gamma_{D}$, $\gamma_{N_{\kappa}}$, $
\gamma_{D}^c$ and $ \gamma_{N_{\kappa}}^c$ can be applied to $\Psi_{E_{\kappa}}$ and
$\Psi_{M_{\kappa}}$, resulting in continuous mappings from
$\HH_{\times}^{-\frac{1}{2}}(\Div_{\Gamma},\Gamma)$ to
itself satisfying
$$
\gamma_{N_{\kappa}}\Psi_{E_{\kappa}} =
\gamma_{D}\Psi_{M_{\kappa}} \quad\textrm{ and }\quad \gamma_{N_{\kappa}}\Psi_{M_{\kappa}} = \gamma_{D}\Psi_{E_{\kappa}}\,.
$$
Defining
$$
 \begin{array}{cclccl}\left[\gamma_{D}\right]&=&\gamma_{D} - \gamma_{D}^c,&\{\gamma_{D}\}&=&-\dfrac{1}{2}\left(\gamma_{D} + \gamma_{D}^c\right),\vspace{2mm}\\
\left[\gamma_{N_{\kappa}}\right]&=&\gamma_{N_{\kappa}} - \gamma_{N_{\kappa}}^c,&
 \{\gamma_{N_{\kappa}}\}&=&-\dfrac{1}{2}\left(\gamma_{N_{\kappa}} + \gamma_{N_{\kappa}}^c\right).
 \end{array}
$$
we have the following \emph{jump relations} (see \cite{BuHiPeSc}):
$$\begin{array}{cclccl}\left[\gamma_{D}\right]\Psi_{E_{\kappa}}&=&0,&\left[\gamma_{N_{\kappa}}\right]\Psi_{E_{\kappa}}&=&-\Id,\vspace{2mm}\\
\left[\gamma_{D}\right]\Psi_{M_{\kappa}}&=&-\Id,&
\left[\gamma_{N_{\kappa}}\right]\Psi_{M_{\kappa}}&=&0.
 \end{array}
$$

Now assume that $\EE\in \LL^2_{\loc}(\R^3)$ belongs to 
$\HH(\Rot,\Omega)$ in the interior domain and to 
$\HH_{\loc}(\Rot,\overline{\Omega^c})$ in the exterior domain
and satisfies the  equation 
\begin{equation}
\label{(*)}
  (\Rot\Rot - \kappa^2\Id)\,\EE = 0
\end{equation}  
in
$\R^3\setminus\Gamma$ and the Silver-M\"uller condition.
Then if we set 
$\jj = [\gamma_{N_{\kappa}}]\EE$, $\mm = [\gamma_{D}]\EE$, we have on $\R^3\setminus\Gamma$ the Stratton-Chu \textbf{integral representation}
\begin{equation}
\label{intrep} 
 \EE = - \Psi_{E_{\kappa}}\jj -\Psi_{M_{\kappa}}\mm.
\end{equation} 
Special cases of \eqref{intrep} are:
If $(\EE^{i},\EE^{s})$ solves the dielectric scattering problem, then
 \begin{equation} \label{(3.1)}-\Psi_{E_{\kappa_{e}}}\gamma_{N_{\kappa_{e}}}^c\EE^{s} -
 \Psi_{M_{\kappa_{e}}}\gamma_{D}^c\EE^{s} = \left\{ \begin{array}{ll} -\EE^{s} & x\in\Omega^c \\ 0 & x\in\Omega
 \end{array}\right.
\end{equation}
\begin{equation} \label{(3.2)}  
 {\Psi_{E_{\kappa_{e}}}\gamma_{N_{\kappa_{e}}}^c\left(\EE^{s}+\EE^{inc}\right) +
\Psi_{M_{\kappa_{e}}}\gamma_{D}^c\left(\EE^{s}+\EE^{inc}\right)} =\left \{ \begin{array}{ll} \EE^{s} &
x\in\Omega^c \\ -\EE^{inc} & x\in\Omega\end{array}\right.
\end{equation}
\begin{equation} \label{(3.3)} 
\qquad{-\Psi_{E_{\kappa_{i}}}\gamma_{N_{\kappa_{i}}}\EE^{i} -
 \Psi_{M_{\kappa_{i}}}\gamma_{D}\EE^{i}} = \left\{ \begin{array}{ll} 0 &  x\in\Omega^c\\ \EE^{i} &
 x\in\Omega\end{array}\right.
\end{equation}

We can now define the main boundary integral operators:
\begin{center}$C_{\kappa} = \{\gamma_{D}\}\Psi_{E_{\kappa}} =
\{\gamma_{N_{\kappa}}\}\Psi_{M_{\kappa}}$,\end{center}
\begin{center}$M_{\kappa} = \{\gamma_{D}\}\Psi_{M_{\kappa}} =
\{\gamma_{N_{\kappa}}\}\Psi_{E_{\kappa}}$.
\end{center}
These are bounded operators in 
$\HH_{\times}^{-\frac{1}{2}}(\Div_{\Gamma},\Gamma)$.

As tools, we will need variants of these operators:
\begin{definition}Define the operators $M_{0}$, $C_{\kappa,0}$ and $C^*_{0}$ for $\jj\in \HH_{\times}^{-\frac{1}{2}}(\Div_{\Gamma},\Gamma)$ by :
$$M_{0}\,\jj=-\{\gamma_{D}\}\Rot\Psi_{0},$$
$$C_{\kappa,0}\,\jj=-\kappa\;\nn\times V_{0}\,\jj+\kappa^{-1}\Rot_{\Gamma}V_{0}\Div_{\Gamma}\jj,$$
$$C_{0}^*\,\jj=\nn\times V_{0}\,\jj+\Rot_{\Gamma}V_{0}\Div_{\Gamma}\jj.$$
\end{definition}
Note that $C_{0}^*$ differs from $C_{1,0}$ by the relative sign of the two terms.

We collect now some properties of these boundary integral operators that are known for Lipschitz domains. 

First we note the following useful relations:
\begin{equation}
\label{eqd2}
  \rot_{\Gamma}\nabla_{\Gamma}=0\text{ and }\Div_{\Gamma}\Rot_{\Gamma}=0
\end{equation}
\begin{equation}
\label{eqd3}
 \Div_{\Gamma}(\nn\times \jj)=-\rot_{\Gamma}\jj\text{ and }\rot_{\Gamma}(\nn\times \jj)=\Div_{\Gamma}\jj\end{equation}

The following lemma is proved in \cite{BuMCSc,BuHiPeSc}.
\begin{lemma}\label{3.4}
(i)
The operators $C_{\kappa} - C_{\kappa,0}$ and $M_{\kappa}-M_{0}$ are compact in
$\HH_{\times}^{-\frac{1}{2}}(\Div_{\Gamma},\Gamma)$. \\
(ii) 
Both $C_{\kappa}$ and $M_{\kappa}$ are antisymmetric with respect to the bilinear form $\mathbf{\mathcal{B}}$.
\end{lemma}
The  Calder\'on projectors for the time-harmonic Maxwell system \eqref{(*)} are $P = \frac{1}{2}\Id +
A_{\kappa}$ and $P^c = \frac{1}{2}\Id - A_{\kappa}$ where 
\begin{displaymath}
A = \left(\begin{array}{ll} M_{\kappa} \ C_{\kappa} \\ C_{\kappa} \ M_{\kappa} \\ \end{array}\right ).
\end{displaymath}
We have $P\circ P^c = 0$ and therefore 
\begin{equation}
\label{Ck2}
  C_{\kappa}^2 = \tfrac{1}{4}\Id - M_{\kappa}^2\,\text{ and }\,C_{\kappa}M_{\kappa}=-M_{\kappa}C_{\kappa}.
\end{equation}
It is a classical result that  when the boundary $\Gamma$ is smooth, the operator $M_{\kappa}$ is compact from $\HH_{\times}^{-\frac{1}{2}}(\Div_{\Gamma},\Gamma)$ to
itself (see \cite{Ned}). From identity \eqref{Ck2} one can then immediately deduce that the ``electric'' operator $C_{\kappa}$ is  Fredholm of index zero. 
The latter result is also true for a Lipschitz boundary $\Gamma$ (see \cite{BuHiPeSc} for more details). Later on, we need the corresponding result for the ``magnetic'' operator $\tfrac12\Id + M_{\kappa}$. This has been proved for Lipschitz domains in \cite[Thm.~4.8]{MiMi02} and in \cite[Thm.~3.2]{Wind}:
\begin{lemma}\label{lMk}
 The operator $\tfrac12\Id + M_{\kappa} : \HH_{\times}^{-\frac{1}{2}}(\Div_{\Gamma},\Gamma)
 \to \HH_{\times}^{-\frac{1}{2}}(\Div_{\Gamma},\Gamma)$
 is Fredholm of index zero.
\end{lemma}
In fact, for $\Im\kappa>0$ this operator is an isomorphism. Later on, we use the result for $\kappa=0$, where we don't know whether it is an isomorphism. But we only need that is an isomorphism up to a compact perturbation, that is, Fredholm of index zero, so we will not pursue this further. 

The following theorem was proved in \cite[Thm.~2.6]{Wind}.
\begin{lemma} \label{elliptic} The operator $ C_{0}^*$ is self-adjoint and elliptic for the bilinear form $\mathbf{\mathcal{B}}$  and invertible on $\HH^{-\frac{1}{2}}_{\times}(\Div_{\Gamma},\Gamma)$. Ellipticity means here that there exists a positive constant $\alpha$ such that for all $\jj\in\HH_{\times}^{-\frac{1}{2}}(\Div_{\Gamma},\Gamma)$
$$
  \mathbf{\mathcal{B}}\big(C_{0}^*\jj,\overline{\jj}\big) \ge \alpha ||\jj||_{\HH_{\times}^{-\frac{1}{2}}(\Div_{\Gamma},\Gamma)}^{2}\,.
$$ 
\end{lemma}
Indeed, for $\jj\in \HH^{-\frac{1}{2}}_{\times}(\Div_{\Gamma},\Gamma)$ we have
\begin{equation}
  \mathbf{\mathcal{B}}(\jj,C_{0}^*\,\overline{\jj}) = 
  \int_\Gamma \left\{\jj\cdot V_0\,\overline{\jj}
  + \Div_{\Gamma}\jj \, V_0 \Div_{\Gamma}\overline{\jj} \right\}
\end{equation}
and the result follows from the $H^{-\frac12}(\Gamma)$-ellipticity of the scalar single layer potential operator $V_0$.

 \section{Integral equations 1}\label{IE1}
 
 In this section, we present the first method for solving the dielectric problem, following the procedure of R. E. Kleinman and P. A. Martin \cite{KlMa}: We use  a layer ansatz on the exterior field to construct two  alternative boundary integral equations. \newline
  In the scalar case, one represents the exterior field as a linear combination of a single layer potential and a double layer potential, both generated by the same density. It turns out that this simple idea does not suffice in the electromagnetic case if one wants to avoid irregular frequencies.  
Our approach is related to the idea of ``modified combined field integral equations'': We compose one of the electromagnetic potential operators with an elliptic and invertible boundary integral operator, namely $C_{0}^*$. 
More precisely, we assume that $\EE^{s}$ admits the following integral representation  :
\begin{equation} 
\label{u2c} 
  \EE^{s}(x) =-a (\Psi_{E_{\kappa_{e}}}\jj)(x) -b(\Psi_{M_{\kappa_{e}}}  C_{0}^*\,\jj)(x)\qquad
  \text{ for }x\in\Omega^c \,.
\end{equation}
Here  $\jj\in\HH_{\times}^{-\frac{1}{2}}(\Div_{\Gamma},\Gamma)$ is the unknown density and $a$ and $b$ are arbitrary complex constants. 

We set $\rho = \dfrac{\mu_{e}\kappa_{i}}{\mu_{i}\kappa_{e}}$. The transmission conditions can be rewritten :
$$\gamma_{D}\EE^{i}=\gamma_{D}^c\EE^{s}+\gamma_{D}\EE^{inc}\textrm{ and }\gamma_{N_{\kappa_{i}}}\EE^{i}=\rho^{-1}\left(\gamma_{N_{\kappa_{e}}}^c\EE^{s}+\gamma_{N_{\kappa_{e}}}\EE^{inc}\right).
$$ 
Using this in the integral representation formula \eqref{(3.3)} in $\Omega$,  we get:
\begin{equation}
\label{u1}  \EE^{i} =
-\frac{1}{\rho}\Psi_{E_{\kappa_{i}}}\left(\gamma_{N_{\kappa_{e}}}^c\EE^{s}+\gamma_{N_{\kappa_{e}}}\EE^{inc}\right) -
\Psi_{M_{\kappa_{i}}}\left(\gamma_{D}^c\EE^{s}+\gamma_{D}\EE^{inc}\right)\text{ in }\Omega.
\end{equation}

We take traces in \eqref{u2c} and obtain the Calder\'on relations
\begin{eqnarray}
\label{Du2}  
\gamma_{D}^c\EE^{s} =& \left\{aC_{\kappa_{e}}-b\left(\tfrac{1}{2}\Id-M_{\kappa_{e}}\right)C_{0}^*\right\}\jj& \equiv
L_{e}\jj\text{ on }\Gamma,\\
\label{Nu2}
\gamma_{N_{\kappa_{e}}}^c\EE^{s} =&
\left\{-a\left(\tfrac{1}{2}\Id-M_{\kappa_{e}}\right)+bC_{\kappa_{e}}C_{0}^*\right\}\jj& \equiv N_{e}\jj\text{ on }\Gamma.
\end{eqnarray}

On the other hand, taking traces in \eqref{u1} gives: 
\begin{eqnarray}\label{eq1} \rho\left(-\frac{1}{2}\Id + M_{\kappa_{i}}\right)\left(\gamma_{D}^c\EE^{s}+\gamma_{D}\EE^{inc}\right) +
C_{\kappa_{i}}\left(\gamma_{N_{\kappa_{e}}}^c\EE^{s}+\gamma_{N_{\kappa_{e}}}\EE^{inc}\right)= 0\text{ on }\Gamma ,\\\label{eq2}
\left(-\frac{1}{2}\Id + M_{\kappa_{i}}\right)\left(\gamma_{N_{\kappa_{e}}}^c\EE^{s}+\gamma_{N_{\kappa_{e}}}\EE^{inc}\right) +
\rho C_{\kappa_{i}}\left(\gamma_{D}^c\EE^{s}+\gamma_{D}\EE^{inc}\right)  = 0\text{ on }\Gamma.\end{eqnarray}

 We can now substitute \eqref{Du2} and \eqref{Nu2} into \eqref{eq1} and get our first integral equation:
\begin{eqnarray}\label{S}
 \SS \jj \equiv \rho\left(-\frac{1}{2}\Id + M_{\kappa_{i}}\right)L_{e}\jj + C_{\kappa_{i}}N_{e}\jj = f \\ 
 \text{ where } 
  f = -\rho\left(-\frac{1}{2}\Id + M_{\kappa_{i}}\right)\gamma_{D} \EE^{inc} -
C_{\kappa_{i}}\gamma_{N_{\kappa_{e}}} \EE^{inc}.
\end{eqnarray}

 If we substitute \eqref{Du2} and \eqref{Nu2} into \eqref{eq2}, we get our second integral equation:
\begin{eqnarray}\label{T} \TT \jj \equiv \rho C_{\kappa_{i}}L_{e}\jj + \left(-\frac{1}{2}\Id + M_{\kappa_{i}}\right)N_{e}\jj = g\\ \text{ where }
 g = -\rho C_{\kappa_{i}}\gamma_{D}\EE^{inc} -
\left(-\frac{1}{2}\Id + M_{\kappa_{i}}\right)\gamma_{N_{\kappa_{e}}}\EE^{inc}.\end{eqnarray}

Thus we obtain two boundary integral equations for the unknown $\jj$. Having solved either one, we construct $\EE^{s}$ using \eqref{u2c} and $\EE^{i}$ using \eqref{u1}, \eqref{Du2}, \eqref{Nu2}:
\begin{equation} \label{u12} 
\EE^{i} =-\frac{1}{\rho}\left(\Psi_{E_{\kappa_{i}}}\{\gamma_{N_{\kappa_{e}}}\EE^{inc} +N_{e}\jj\}\right) -
\left(\Psi_{M_{\kappa_{i}}}\{\gamma_{D}\EE^{inc} + L_{e}\jj\}\right).
\end{equation}

\begin{theorem}
\label{solpt} 
 If $\jj \in \HH_{\times}^{-\frac{1}{2}}(\Div_{\Gamma},\Gamma)$ solves \eqref{S} or \eqref{T}, then $\EE^{s}$ and $\EE^{i}$ given by \eqref{u2c} and \eqref{u12} solve the transmission problem.\end{theorem}
 
 \begin{proof} We know that $\EE^{i}$ and $\EE^{s}$ satisfy the Maxwell equations and the Silver-M\"uller condition. It remains to verify that  $\EE^{s}$ and  $\EE^{i}$ satisfy the transmission conditions \eqref{T1} and \eqref{T2}. Using the integral representation \eqref{u2c} and \eqref{u1} of $\EE^{s}$ and $\EE^{i}$, a simple computation gives:
\begin{equation} 
\label{w} 
 \begin{array}{ll} \vspace{1mm}\;\;\;\rho(\gamma_{D}^c \EE^{s}+ \gamma_{D}\EE^{inc} -
 \gamma_{D}\EE^{i}) =\SS\jj-f
 \end{array}
 \end{equation}
 and
\begin{equation}
\label{ww} 
 \gamma_{N_{\kappa_{e}}}^c\EE^{s} + \gamma_{N_{\kappa_{e}}}\EE^{inc} -
 \rho\gamma_{N_{\kappa_{i}}}\EE^{i} = \TT\jj - g
\end{equation}
 
 We deduce that\\
  - if $\jj$ solves \eqref{S}, then relation \eqref{w} proves that  the condition \eqref{T1}  is
 satisfied,\\
 - if $\jj$ solves \eqref{T}, then relation \eqref{ww} proves that the condition \eqref{T2} is
 satisfied. 
 
Now we show that \eqref{S} and \eqref{T} are in fact equivalent. 
Define :\begin{center}$\uu(x) =
 -\Psi_{E_{\kappa_{i}}}\{\gamma_{N_{\kappa_{e}}}\EE^{inc} + N_{e}\jj\}(x) -
 \rho\Psi_{M_{\kappa_{i}}}\{\gamma_{D}\EE^{inc} + L_{e}\jj\}(x) \quad \textrm{ for }
 x\in\Omega^c$.\end{center}
  This field $\uu$ is in $\HH_{\loc}(\Rot,\overline{\Omega^c})$ and satisfies the Maxwell system
 
 \begin{equation}\label{meki}\Rot\Rot\uu-\kappa_{i}^2\uu=0\end{equation} in $\Omega^c$. On the boundary $\Gamma$ we have:
$$\gamma_{D}^c\uu = \SS\jj-f\quad\text{ and }\quad \gamma_{N_{\kappa_{i}}}^c\uu=\TT\jj-g.
$$
 Since $\uu$ solves \eqref{meki} in $\Omega^c$ and satisfies the
 Silver-M\"uller condition, it follows:
 $$
 \jj\text{ satisfies }\eqref{S}\Rightarrow\gamma_{D}^c\uu = 0\Rightarrow \uu\equiv0
 \text{ in }\overline{\Omega}^c\Rightarrow\gamma_{N_{\kappa_{i}}}^c\uu=0\Rightarrow \jj \text{ satisfies }\eqref{T}.$$
 $$\jj\text{ satisfies }\eqref{T}\Rightarrow\gamma_{N_{\kappa_{i}}}^c\uu = 0\Rightarrow \uu\equiv0
 \text{ in }\Omega^c\Rightarrow\gamma_{D}^c\uu=0\Rightarrow \jj \text{ satisfies }\eqref{S}.
$$ 
As a consequence, if $\jj$ solves one of the two integral equations, it solves both, and then both transmission conditions \eqref{T1} and \eqref{T2} are satisfied. 
\end{proof}

  The next theorem is concerned with the uniqueness of the solutions of the boundary  integral equations \eqref{S}
and \eqref{T}, i.e., with the existence of nontrivial solutions of the following homogeneous forms of \eqref{S}
and \eqref{T}:
\begin{eqnarray}\label{S0}\rho\left(-\tfrac{1}{2}\Id +
M_{\kappa_{i}}\right)L_{e}\jj_{0} + C_{\kappa_{i}}N_{e}\jj_{0} = 0,\\\label{T0}
 \rho C_{\kappa_{i}}L_{e}\jj_{0} + \left(-\frac{1}{2}\Id + M_{\kappa_{i}}\right)N_{e}\jj_{0}
= 0.\end{eqnarray}

 As in the scalar case \cite{KlMa}, we associate with the dielectric scattering problem a new interior boundary value problem, the eigenvalues of which determine uniqueness for the integral equations.
 
\medskip 

\noindent
\textbf{Associated interior problem:} 
For $a,b\in\C$, consider the boundary value problem
\begin{equation}
\label{intp} 
 \Rot\Rot \uu - \kappa_{e}^{2}\uu = 0 \quad\text{ in } \Omega, 
 \qquad 
 a\gamma_{D}\uu - bC_{0}^*\gamma_{N_{e}}\uu = 0 \quad\text{ on }\Gamma. 
\end{equation}

\begin{lemma} 
\label{lem}.
Let $a,b\in\C\setminus\{0\}$ and let $\kappa_{e}\in\C$. Assume that
\begin{itemize}
\item[(i)] $\Im\left(\dfrac{a}{b}\right)\not=0\,$ if $\kappa_{e}\in\R$,
 \item[(ii)] $\Im(\kappa_{e}^2)\cdot\Im\left(\kappa_{e}\dfrac{a}{b}\right)>0\,$ if $\kappa_{e}\in\C\backslash\R$,\end{itemize}
  then $\kappa_{e}^2$ is not an eigenvalue of the  the associated interior problem \eqref{intp}.
\end{lemma}
 
  \begin{proof}
Let $\kappa_{e}^2$ be an eigenvalue of the interior problem  and let $\uu\not=0$ be an eigenfunction. Using Green's theorem we have:
\begin{equation*}
 \begin{split}
  \int_{\Omega}|\Rot \uu|^2-\kappa_{e}^2\int_{\Omega}|\uu|^2=\kappa_{e}\,\mathcal{B}(\gamma_{D}\uu,\gamma_{N_{\kappa_{e}}}\overline{\uu})&=\overline{\kappa}_{e}\,\mathcal{B}(\gamma_{D}\uu,\overline{\gamma_{N_{\kappa_{e}}}\uu})\\&=\overline{\left(\kappa_{e}\frac{a}{b}\right)} \;\mathcal{B}(\gamma_{D} \uu,(C_{0}^{*})^{-1}\overline{\gamma_{D}\uu})\mathrm{\;if\;b\not=0}
  \\&=\overline{\kappa_{e}}\frac{b}{a}\mathcal{B}(C_{0}^{*}\gamma_{N_{\kappa_{e}}}\uu,\overline{\gamma_{N_{\kappa_{e}}}\uu})\mathrm{\;if\;a\not=0}
 \end{split}
\end{equation*}
Since $C_{0}^*$ is elliptic for the bilinear form $\mathcal{B}$, taking the imaginary part, we obtain 
$$\begin{array}{ll}-\Im(\kappa_{e}^2)\displaystyle{\int_{\Omega}|\uu|^2}&=-\Im\overline{\left(\kappa_{e}\dfrac{a}{b}\right)}\mathcal{B}((C_{0}^{*})^{-1}\gamma_{D} \uu,\overline{\gamma_{D}\uu})\vspace{2mm}\\&=-|\kappa_{e}|^2\Im\left(\dfrac{b}{\kappa_{e}a}\right)\mathcal{B}(\gamma_{N_{\kappa_{e}}}\uu,C_{0}^{*}\overline{\gamma_{N_{\kappa_{e}}}\uu}).\end{array}$$
Under the hypotheses of the lemma the left-hand side and the right-hand side have opposite sign, and it follows 
$$ 
 \mathcal{B}((C_{0}^{*})^{-1}(\gamma_{D} \uu),\overline{\gamma_{D}\uu})=0\text{ and }\mathcal{B}(\gamma_{N_{\kappa_{e}}}\uu,C_{0}^{*}\overline{\gamma_{N_{\kappa_{e}}}\uu})=0.
$$ 
As $C_{0}^{*}$ is elliptic for the bilinear  form $\mathcal{B}$, the traces $\gamma_{D}\uu$ and $\gamma_{N_{\kappa_{e}}}\uu$ then vanish. Thanks to the Stratton-Chu representation formula \eqref{intrep} in $\Omega$, we deduce that $\uu=0$,  which contradicts the initial assumption.
\end{proof}

\begin{remark}
Note that this associated interior problem is not an impedance problem (or Robin problem) as in the scalar case \cite{KlMa}. If we replace in \eqref{intp} the operator $C_{0}^{*}$ by the identity, we obtain a ``pseudo-impedance'' type problem. This is a non-elliptic problem, about whose spectrum we have no information. That the problem is non-elliptic can be seen as follows: If it were elliptic, its principal part would be elliptic, too. This would be the vector Laplace operator with the ``Neumann'' condition $\gamma_{N_{\kappa_{e}}}\uu=0$. Any gradient of a harmonic function in $H^1(\Omega)$ will satisfy the homogeneous problem, which therefore has an infinite-dimensional nullspace, contradicting ellipticity. Note that the issue here is not the apparent non-elliptic nature of the interior Maxwell $\Rot\Rot$ operator, which can easily be remedied by the usual regularization that adds $-\nabla\Div$, but the manifestly non-elliptic nature of the Maxwell ``Neumann'' boundary operator. For a ``true'' impedance problem, the operator $C_{0}^{*}$ would have to be replaced not by the identity, but by the rotation operator $\jj\mapsto\nn\times\jj$ which is used in Mautz's formulation \cite{Mautz}. This operator leads out of the space $\HH_{\times}^{-\frac{1}{2}}(\Div_{\Gamma},\Gamma)$, however, which rules it out for our purposes.\end{remark}

For our integral equations, the problem \eqref{intp} plays the same role as the Robin problem for the scalar case in \cite{KlMa}.

\begin{theorem} 
\label{vp} 
Assume that the hypotheses of Theorem~\ref{t1} are satisfied. 
Then
for $(a,b)\ne(0,0)$,
the homogeneous integral equations \eqref{S0} and \eqref{T0} admit  nontrivial solutions if and only if $\kappa_{e}^2$ is an eigenvalue of the associated interior problem.
\end{theorem}

\begin{proof} Assume that $\jj_{0} \not = 0$ solves \eqref{S0}
or \eqref{T0}.\newline 
We construct $\uu_{2}$ and $\uu_{1}$ as follows:
\begin{displaymath}
\begin{array}{ll}\vspace{1mm} \uu_{2}(x) = -a\Psi_{E_{\kappa_{e}}}\jj_{0}(x) -
b\Psi_{M_{\kappa_{e}}}C_{0}^* \jj_{0}(x)\textrm{ for } x\in\Omega^c \\ \uu_{1}(x) =
-\dfrac{1}{\rho}\Psi_{E_{\kappa_{i}}}(N_{e}\jj_{0})(x) - \Psi_{M_{\kappa_{i}}}(L_{e}\jj_{0})(x)\textrm{ for }
  x\in \Omega \end{array}
\end{displaymath}
By Theorem \ref{solpt}, $\uu_{1}$ and $\uu_{2}$ together solve the transmission problem
 with $\EE^{inc} = 0$.\newline
 Since this problem admits at most one solution, we have $\uu_{2} \equiv 0$ in
$\Omega^c$ and $\uu_{1} \equiv 0$ in $\Omega$.\newline
 Now we set $\uu(x) = -a\Psi_{E_{\kappa_{e}}}\jj_{0}(x) - b\Psi_{M_{\kappa_{e}}}C_{0}^* \jj_{0}(x)$ for $x\in\Omega$.
\newline 
We have on $\Gamma$ :
\begin{eqnarray}
\label{x} 
  \gamma_{D}^c\uu_{2} - \gamma_{D}\uu = b\,C_{0}^* \jj_{0}, \\
\label{xx}
 \gamma_{N_{\kappa_{e}}}^c\uu_{2} - \gamma_{N_{\kappa_{e}}}\uu = a\jj_{0}. \end{eqnarray}
 Since $\gamma_{D}^c\uu_{2} = \gamma_{N_{e}}^c \uu_{2}= 0$ on $\Gamma$, we find
\begin{equation*} 
 a\gamma_{D}\uu - bC_{0}^*\gamma_{N_{\kappa_{e}}}\uu = 0\text{ on }\Gamma.\end{equation*}
  Thus $\uu$ is an eigenfunction associated with the eigenvalue $\kappa_{e}^2$ of the interior problem
 or $\uu\equiv 0$. But this latter possibility can be eliminated
since it implies that $\gamma_{D}\uu =\gamma_{N_{\kappa_{e}}}\uu
=0$, whence $\jj_{0} = 0$ by \eqref{x} and \eqref{xx}, which is contrary to the assumption.\newline
Conversely, assume that $\kappa_{e}^2$ is an eigenvalue of the associated interior problem. Let $v_{0}\not\equiv 0$ be a corresponding eigenfunction. The Calder\'on relations in $\Omega$ imply that :
\begin{displaymath}\begin{array}{ll}
-C_{\kappa_{e}}\gamma_{N_{\kappa_{e}}}v_{0} + \left(\tfrac{1}{2}\Id - M_{\kappa_{e}}\right)\gamma_{D}v_{0} = 0, \\
\left(\tfrac{1}{2}\Id - M_{\kappa_{e}}\right)\gamma_{N_{\kappa_{e}}}v_{0} - C_{\kappa_{e}}\gamma_{D}v_{0} = 0.
\end{array}
\end{displaymath}
 Using the equality $a\gamma_{D}v_{0}-b\;C_{0}^*\gamma_{N_{\kappa_{e}}}v_{0}=0$,  we obtain 
$$
 L_{e}(C_{0}^*)^{-1}\gamma_{D}v_{0} = 0, \quad
 N_{e}(C_{0}^*)^{-1}\gamma_{D}v_{0} = 0, \quad
 L_{e}\gamma_{N_{\kappa_{e}}}v_{0} = 0, \quad
 N_{e}\gamma_{N_{\kappa_{e}}}v_{0} = 0.
$$
If $b\not = 0$, then $\gamma_{D}v_{0}\ne0$, and
$\jj_0=(C_{0}^*)^{-1}\gamma_{D}v_{0}$ is a nontrivial solution of \eqref{S0} and \eqref{T0}.\newline
If $b = 0$, then $\gamma_{N_{\kappa_{e}}}v_{0}\ne0$, and 
$\jj_0=\gamma_{N_{\kappa_{e}}}v_{0}$ is a nontrivial solution of \eqref{S0}
and \eqref{T0}.
\end{proof}

Whereas until now we only assumed that the boundary $\Gamma$ is Lipschitz, we now present two theorems on the operators $\SS$ and $\TT$ that are, in this generality, only valid for smooth boundaries. By \emph{smooth} we mean, for simplicity, $\mathscr{C}^{\infty}$ regularity, although a careful check of the proof would show that some finite regularity, such as $\mathscr{C}^{2}$, would be sufficient.  

\begin{theorem}\label{thT1}
 Assume that 
 \begin{itemize}
 \item[(i)] the boundary $\Gamma$ is smooth and simply connected,
 \item[(ii)] the constants $a$, $b$, $\mu_{e}$, $\mu_{i}$, $\kappa_{e}$ and $\kappa_{i}$ satisfy:
$$ 
 \begin{array}{ll}\left(b\kappa_{e}+2a\right)\not=0,\;\left(1+\dfrac{\mu_{e}}{\mu_{i}}\right)\not=0,\;\left(b-2a\kappa_{e}\right)\not=0\text{ and }\left(1+\dfrac{\mu_{e}\kappa_{i}^2}{\mu_{i}\kappa_{e}^2}\right)\not=0.\end{array}
$$\end{itemize}
Then $\SS$ is a  Fredholm operator of index zero
on $\HH_{\times}^{-\frac{1}{2}}(\Div_{\Gamma},\Gamma)$.
\end{theorem}

\begin{proof}\label{prT1}
We can rewrite $\SS$ as follows:
\begin{displaymath}
\begin{array}{ll}\SS =& \tfrac{1}{4}b\rho C_{0}^*-
\tfrac{1}{2}b\rho(M_{\kappa_{i}}+M_{\kappa_{e}})C_{0}^* +b\rho M_{\kappa_{i}}M_{\kappa_{e}}C_{0}^*-\tfrac{1}{2}a\rho(C_{\kappa_{e}}-C_{\kappa_{e},0})\vspace{2mm}\\&-\tfrac{1}{2}a(C_{\kappa_{i}}-C_{\kappa_{i},0})-\tfrac{1}{2}a(\rho C_{\kappa_{e},0}+C_{\kappa_{i},0})+  a\rho
M_{\kappa_{i}}C_{\kappa_{e}} +  aC_{\kappa_{i}}M_{\kappa_{e}}\vspace{2mm}\\&+b(C_{\kappa_{i}}-C_{\kappa_{i},0})C_{\kappa_{e}}C_{0}^*+bC_{\kappa_{i},0}(C_{\kappa_{e}}-C_{\kappa_{e},0})C_{0}^*+
bC_{\kappa_{i},0}C_{\kappa_{e},0}C_{0}^*.
\end{array}
\end{displaymath}
Thus $\SS$ is a compact perturbation of the operator 
$$\SS_{1}=b\left(\tfrac{1}{4}\rho \Id+C_{\kappa_{i},0}C_{\kappa_{e},0}\right)C_{0}^*-\frac{1}{2}a\left(\rho C_{\kappa_{e},0}+C_{\kappa_{i},0}\right).$$
We have to show that the operator $\SS_{1}$ is Fredholm of index zero. 
For this we use the \emph{Helmholtz decomposition} of $\HH_{\times}^{-\frac{1}{2}}(\Div_{\Gamma},\Gamma)$ :
\begin{equation}
\label{Hholtz}
 \HH_{\times}^{-\frac{1}{2}}(\Div_{\Gamma},\Gamma)= 
 \nabla_{\Gamma} H^{\frac{3}{2}}(\Gamma) \oplus
 \Rot_{\Gamma} H^{\frac{1}{2}}(\Gamma)\,.
\end{equation}
For a detailed proof of \eqref{Hholtz} see \cite{dlB}. Note that we are assuming that the boundary $\Gamma$ is smooth and simply connected. For a proof of the following result, we refer to \cite{BuCi2,MartC,Ned}.
\begin{lemma}
\label{LapBel}
 Let $\Gamma$ be smooth and simply connected and $t\in\R$. The Laplace-Beltrami operator 
 \begin{equation}\label{eqd1}
 \Delta_{\Gamma}=\Div_{\Gamma}\nabla_{\Gamma}=-\rot_{\Gamma}\Rot_{\Gamma}\end{equation} 
is linear and continuous from $H^{t+2}(\Gamma)$ to $H^t(\Gamma)$. \\
It is an isomorphism from $H^{t+2}(\Gamma)\slash\R$ to the space
$H^t_{*}(\Gamma)$ defined by
$$
u\in H^t_{*}(\Gamma)\quad\Longleftrightarrow\quad u\in H^t(\Gamma)\textrm{ and }\int_{\Gamma}u=0.
$$
\end{lemma}
The terms in the decomposition 
$\jj=\nabla_{\Gamma}p+\Rot_{\Gamma}q$ for 
$\jj\in  \HH_{\times}^{-\frac{1}{2}}(\Div_{\Gamma},\Gamma)$
are obtained by solving the Laplace-Beltrami equation:
$$
  p=\Delta_{\Gamma}^{-1}\Div_{\Gamma}\jj\;,\qquad
  q=-\Delta_{\Gamma}^{-1}\rot_{\Gamma}\jj\,.
$$
The mapping 
\begin{equation}
 \begin{array}{ccc}\HH^{-\frac{1}{2}}_{\times}(\Div_{\Gamma},\Gamma)&\rightarrow& H^{\frac{3}{2}}(\Gamma)/\R\times H^{\frac{1}{2}}(\Gamma)/\R\\\jj=\nabla_{\Gamma}p+\Rot_{\Gamma}q&\mapsto&\left(\begin{array}{cc}p\\q\end{array}\right)
 \end{array}
\end{equation}
is an isomorphism. 
Using this isomorphism, we can rewrite the operator $\SS_{1}$ as an operator $\mathcal{S}_{1}$ defined from $H^{\frac{3}{2}}(\Gamma)\slash\R\times H^{\frac{1}{2}}(\Gamma)\slash\R$ into itself. Then to show that $\SS_{1}$ it is Fredholm of index zero it suffices to show that $\mathcal{S}_{1}$ has this property.
 Let us begin by rewriting $C_{0}^{*}$ and $C_{\kappa,0}$ as operators $\mathcal{C}_{0}^{*}$ and $\mathcal{C}_{\kappa,0}$ defined on $H^{\frac{3}{2}}(\Gamma)\slash\R\times H^{\frac{1}{2}}(\Gamma)\slash\R$. 
 We have to determine $P_{0}\in H^{\frac{3}{2}}(\Gamma)\slash\R$ and $Q_{0}\in H^{\frac{1}{2}}(\Gamma)\slash\R$ such that $C_{0}^{*}(\nabla_{\Gamma}p+\Rot_{\Gamma}q)=\nabla_{\Gamma}P_{0}+\Rot_{\Gamma}Q_{0}$, and this defines $\mathcal{C}_{0}^{*}$ by: 
$$
 \mathcal{C}_{0}^{*}\left(\begin{array}{cc}p\\q\end{array}\right)=\left(\begin{array}{cc}P_{0}\\Q_{0}\end{array}\right).
$$  
We have 
$$
 P_{0}=\Delta_{\Gamma}^{-1}\Div_{\Gamma}C_{0}^{*}(\nabla_{\Gamma}p+\Rot_{\Gamma}q)$$ and  $$Q_{0}=-\Delta_{\Gamma}^{-1}\rot_{\Gamma}C_{0}^{*}(\nabla_{\Gamma}p+\Rot_{\Gamma}q).
$$ 
Using the integral representation of $C_{0}^{*}$ and the equalities \eqref{eqd2} and \eqref{eqd3} we obtain:
\begin{equation*}
\mathcal{C}_{0}^{*}=\left(\begin{array}{cc}\mathcal{C}_{11}&\mathcal{C}_{12}\\\mathcal{C}_{21,1}+\mathcal{C}_{21,2}&\mathcal{C}_{22}\end{array}\right),
\end{equation*}
where
$$
\begin{array}{lclccl}C_{11}&=&-\Delta_{\Gamma}^{-1}\rot_{\Gamma}V_{0}\nabla_{\Gamma},&C_{12}&=&-\Delta_{\Gamma}^{-1}\rot_{\Gamma}V_{0}\Rot_{\Gamma},\\C_{21,1}&=&-\Delta_{\Gamma}^{-1}\Div_{\Gamma}V_{0}\nabla_{\Gamma},&C_{22}&=&-\Delta_{\Gamma}^{-1}\Div_{\Gamma}V_{0}\Rot_{\Gamma},\\C_{21,2}&=&\;V_{0}\Delta_{\Gamma}.&&&
\end{array}
$$
Some of these operators are of lower order than what a simple counting of orders (with -1 for the order of $V_0$) would give:
\begin{lemma} \label{regular}
 Let $\Gamma$ be smooth. Then 
 the operators $\rot_{\Gamma}V_{0}\nabla_{\Gamma}$ and $\Div_{\Gamma}V_{0}\Rot_{\Gamma}$ are linear and continuous from 
 $H^t(\Gamma)$ into itself.
\end{lemma}
\begin{proof}
 These results are due to the equalities \eqref{eqd2}.
 One can write (see \cite[page 240]{Ned}):
\begin{eqnarray*}  
 \rot_{\Gamma}V_{0}\nabla_{\Gamma}u(x)&=&\displaystyle{\int_{\Gamma}\nn(x)\cdot\Rot^x\left\{G(0,|x-y|)\nabla_{\Gamma}u(y)\right\}d\sigma(y)}\vspace{2mm}\\&=&\displaystyle{\int_{\Gamma}\left\{\left(\nn(x)-\nn(y)\right)\times\nabla^xG(0,|x-y|)\right\}\cdot\nabla_{\Gamma}u(y)d\sigma(y)}\vspace{2mm}\\&&-V_{0}\rot_{\Gamma}\nabla_{\Gamma}u.
\end{eqnarray*}
The second term on the right hand side vanishes, and the kernel 
$$
 \left(\nn(x)-\nn(y)\right)\times\nabla^xG(0,|x-y|)
$$ 
has the same weak singularity as the fundamental solution $G(0,|x-y|)$. We deduce the lemma using similar arguments for the other operator. 
\end{proof}

As a consequence, the operators $\mathcal{C}_{11}$ and $\mathcal{C}_{22}$ are of order -2, the operators $\mathcal{C}_{12}$ and $\mathcal{C}_{21,1}$ are of order -1 and the operator $C_{21,2}$ is of order 1. 
Therefore, $\mathcal{C}_{0}^*$ is a compact perturbation of
$$
  \left(\begin{array}{cc}0&\mathcal{C}_{12}\\
    \mathcal{C}_{21,2}&0
  \end{array}\right)
$$

By definition of $C_{\kappa,0}$, the operator $\mathcal{C}_{\kappa,0}$ can be written as: 
\begin{equation*}
 \begin{array}{ccl}
   \mathcal{C}_{\kappa,0}&=&\left(
    \begin{array}{cc}-\kappa \mathcal{C}_{11}&-\kappa \mathcal{C}_{12}\\-\kappa \mathcal{C}_{21,1}+\kappa^{-1}\mathcal{C}_{21,2}&-\kappa \mathcal{C}_{22}
    \end{array}\right)\vspace{2mm}\\
    &=&\left(\begin{array}{ll}-\kappa&0\\0&\kappa^{-1}\end{array}\right)
    \mathcal{C}_{0}^*-(\kappa+\kappa^{-1})\left(
    \begin{array}{ll}0&0\\\mathcal{C}_{21,1}&\mathcal{C}_{22}
    \end{array}\right).
 \end{array}
\end{equation*}
The second term on the right hand side is compact on $H^{\frac{3}{2}}(\Gamma)\slash\R\times H^{\frac{1}{2}}(\Gamma)\slash\R$.\newline

Since $\mathcal{C}_{\kappa,0}$ is a  compact perturbation of 
$$
 \left(\begin{array}{ll}-\kappa&0\\0&\kappa^{-1}\end{array}\right)\mathcal{C}_{0}^{*},
$$ 
 the sum  $\mathcal{C}_{\kappa_{i},0}+\rho \mathcal{C}_{\kappa_{e},0}$ is a compact perturbation of  
$$
\left(\begin{array}{cc}-(\kappa_{i}+\rho\kappa_{e})&0\\0&(\kappa_{i}^{-1}+\rho\kappa_{e}^{-1})\end{array}\right)\mathcal{C}_{0}^{*}.
$$ 
 The operator $\mathcal{C}_{\kappa_{i},0}\mathcal{C}_{\kappa_{e},0}$ is a compact perturbation of  $$\left(\begin{array}{cc}-\kappa_{i}\kappa_{e}^{-1}&0\\0&-\kappa_{i}^{-1}\kappa_{e}\end{array}\right)\mathcal{C}^{*2}_{0}.$$

 \begin{remark} Notice that lemma \ref{regular} is not true in the Lipschitz case. Nevertheless, one can use the Helmholtz decomposition also for a Lipschitz boundary. One only has to replace the space $H^{\frac{3}{2}}(\Gamma)$ by the more general space
$$
  \mathcal{H}(\Gamma) = \{u\in H^{1}(\Gamma),\; \Delta_{\Gamma}u\in H^{-\frac12}(\Gamma)\}\,,
$$
see {\rm \cite{BuMCSc,BuHiPeSc}}. Then a large part of the previous arguments is still valid. For example, 
 the operators $C_{11}$ and $C_{22}$, being of order $-1$, are still compact from $\mathcal{H}(\Gamma)$ and $H^{\frac{1}{2}}(\Gamma)$, respectively, to themselves. By the compactness of the embedding $\mathcal{H}(\Gamma)\hookrightarrow H^{\frac{1}{2}}(\Gamma)$ we deduce that the operator $C_{21,1}$ is still compact from $\mathcal{H}(\Gamma)$ to $H^{\frac{1}{2}}(\Gamma)$. The complete proof of the theorem does not go through, however, because it uses the compactness of $M_{\kappa}$.
\end{remark}

\begin{lemma} 
 For smooth $\Gamma$, the operator $C_{0}^{*2}$ is a compact perturbation of 
$-\tfrac{1}{4}\Id$. 
\end{lemma}
\begin{proof} It suffices to consider the principal part of \eqref{Ck2}. \end{proof}

Collecting all the results, we find that  $\mathcal{S}_{1}$ is a compact perturbation of
\begin{equation}\label{estim1}\begin{split}\left(\begin{array}{cc}\tfrac{1}{4}b(\rho+\kappa_{i}\kappa_{e}^{-1})-\tfrac{1}{2}a(\kappa_{i}+\rho\kappa_{e})&0\\0&\tfrac{1}{4}b(\rho+\kappa_{i}^{-1}\kappa_{e})+\tfrac{1}{2}a(\kappa_{i}^{-1}+\rho\kappa_{e}^{-1})\end{array}\right)\mathcal{C}_{0}^{*}.\end{split}\end{equation}
\bigskip We recall that  $\rho=\dfrac{\mu_{e}\kappa_{i}}{\mu_{i}\kappa_{e}}$. The matrix written above is invertible if: 
$$
 \frac{1}{4}b(\rho+\kappa_{i}\kappa_{e}^{-1})-\frac{1}{2}a(\kappa_{i}+\rho\kappa_{e})\not=0\quad\Leftrightarrow\quad\frac{1}{4}(b-2a\kappa_{e})\left(1+\dfrac{\mu_{i}}{\mu_{e}}\right)\not=0
$$
and
$$\frac{1}{4}b(\rho+\kappa_{i}^{-1}\kappa_{e})+\frac{1}{2}a(\kappa_{i}^{-1}+\rho\kappa_{e}^{-1})\not=0\quad\Leftrightarrow\quad\frac{1}{4}(b\kappa_{e}+2a)\left(1+\dfrac{\mu_{i}\kappa_{e}^2}{\mu_{e}\kappa_{i}^2}\right)\not=0.
$$
Since the operator $C_{0}^{*}$ is invertible, we conclude that under the conditions of the theorem  the operator $\SS_{1}$ is  Fredholm of index zero and therefore $\SS$ too. The theorem is proved.
\end{proof}

Using similar arguments we obtain the following theorem.
\begin{theorem} 
\label{thT2}
 Assume that
  \begin{itemize}
 \item[(i)] the boundary $\Gamma$ is smooth and simply connected,
 \item[(ii)] the constants $a$, $b$, $\mu_{e}$, $\mu_{i}$, $\kappa_{e}$ and $\kappa_{i}$ satisfy:
$$ 
 \begin{array}{ll}\left(b\kappa_{e}-2a\right)\not=0,\;\left(1+\dfrac{\mu_{e}}{\mu_{i}}\right)\not=0,\;\left(b+2a\kappa_{e}\right)\not=0\text{ and }\left(1+\dfrac{\mu_{e}\kappa_{i}^2}{\mu_{i}\kappa_{e}^2}\right)\not=0.\end{array}
$$\end{itemize}
  Then $\TT$ is a Fredholm operator of index zero in $\HH_{\times}^{-\frac{1}{2}}(\Div_{\Gamma},\Gamma)$.
\end{theorem}
\begin{proof} The operator $\TT$ is a compact perturbation of $$\TT_{1}=a\left(\frac{1}{4} I+\rho C_{\kappa_{i},0}C_{\kappa_{e},0}\right)-\frac{b}{2}\left( C_{\kappa_{e},0}+\rho C_{\kappa_{i},0}\right)C^*_{0}.$$
We use again the Helmholtz decomposition of $\HH_{\times}^{-\frac{1}{2}}(\Div_{\Gamma},\Gamma)$ and rewrite the operators $C_{\kappa,0}$ as operators defined on $H^{\frac{3}{2}}(\Gamma)\slash\R\times H^{\frac{1}{2}}(\Gamma)\slash\R$. Collecting the results from the previews proof we found that  the term  $\big(\mathcal{C}_{\kappa_{e},0}+\rho \mathcal{C}_{\kappa_{i},0}\big)\mathcal{C}_{0}^*$ is a compact perturbation of  
$$
\frac{1}{4}\left(\begin{array}{cc}(\kappa_{e}+\rho\kappa_{i})&0\\0&-(\kappa_{e}^{-1}+\rho\kappa_{i}^{-1})\end{array}\right).$$
Finally the operator $\TT_{1}$ is a compact perturbation of 
\begin{equation}\label{estim2}\frac{1}{4}\left(\begin{array}{cc}a\left(1+\rho\kappa_{i}\kappa^{-1}_{e}\right)-\tfrac{1}{2}b\left(\kappa_{e}+\rho\kappa_{i}\right)&0\\0&a\left(1+\rho\kappa_{e}\kappa_{i}^{-1}\right)+\tfrac{1}{2}b\left(\kappa_{e}^{-1}+\rho\kappa_{i}^{-1}\right)\end{array}\right).\end{equation}
\end{proof}
Note that under standard hypotheses on the materials and for real frequencies, the material factors such as
$\bigl(1+\dfrac{\mu_{e}}{\mu_{i}}\bigr)$ and $\bigl(1+\dfrac{\mu_{e}\kappa_{i}^2}{\mu_{i}\kappa_{e}^2}\bigr)$ are always non-zero.

\begin{remark} Thanks to the explicit representations \eqref{estim1} and \eqref{estim2} one can deduce G\r{a}rding inequalities (positivity modulo compact perturbations) in $\HH_{\times}^{-\frac{1}{2}}(\Div_{\Gamma},\Gamma)$ for  $\SS$ (via the bilinear form $\mathcal{B}$) and for $\TT$ (via the $\LL^2$-duality pairing) in the case of a domain diffeomorphic to a ball.
\end{remark}

When $\Gamma$ is a only a Lipschitz continuous surface, one can still prove that the operators $\SS$ and $\TT$ are Fredholm operators of index zero, if one imposes more restrictive hypotheses on the physical parameters. We have the following G\r{a}rding inequalities.
\begin{theorem}\label{ThLip}Assume that
\begin{itemize}
\item[(i)] $\mu_{e}$, $\mu_{i}$, $\kappa_{e}$ and $\kappa_{i}$ are  positive real numbers.
\item[(ii)] $a=1$ and $b=-i\eta$ with $\eta\in\R$, $\eta>0$,
\end{itemize}
Then the operators $\SS$ and $\TT$ satisfy the following G\r{a}rding inequalities
$$
\begin{aligned}
\Im\left(\mathcal{B}(\SS\jj,\overline{\jj})+c_{S}(\jj,\jj)\right)&\ge C_{S}||\jj||_{\HH_{\times}^{-\frac{1}{2}}(\Div_{\Gamma},\Gamma)}\\
\Re\left(\mathcal{B}(\TT C_{0}^{*-1}\jj,\overline{\jj})+c_{T}(\jj,\jj)\right)&\ge C_{T}||\jj||_{\HH_{\times}^{-\frac{1}{2}}(\Div_{\Gamma},\Gamma)}
\end{aligned}
$$
 where $c_{S}$ and $c_{T}$ are compact bilinear forms and $C_{S}$ and $C_{T}$ are positive real constants.
\end{theorem}
\begin{proof} According to the definitions \eqref{Du2}--\eqref{S} and Lemma~\ref{3.4}, the operator $\SS$ is a compact perturbation of 
$$\SS_{1}=-i\eta\left(\rho\big(-\tfrac{1}{2}\Id+M_{0}\big)^2\!\!+\!C_{\kappa_{i},0}C_{\kappa_{e},0}\right)C_{0}^*-\tfrac{1}{2}\left(\rho C_{\kappa_{e},0}+C_{\kappa_{i},0}\right)+\rho C_{\kappa_{e},0}M_{0}+M_{0}C_{\kappa_{i},0}.
$$
Let $\jj\in\HH_{\times}^{-\frac{1}{2}}(\Div_{\Gamma},\Gamma)$. The term $\mathcal{B}(\left(\rho C_{\kappa_{e},0}+C_{\kappa_{i},0}\right)\jj,\overline{\jj})$ is real. We also have 
$$\mathcal{B}(C_{\kappa_{e},0}M_{0}\jj,\overline{\jj})=-\mathcal{B}(M_{0}\jj,C_{\kappa_{e},0}\overline{\jj})=\mathcal{B}(\jj,M_{0}C_{\kappa_{e},0}\overline{\jj})=-\mathcal{B}(M_{0}C_{\kappa_{e},0}\overline{\jj},\jj)
$$
The term $M_{0}C_{\kappa_{e},0}$ is a compact perturbation of $M_{\kappa_{e}}C_{\kappa_{e}}$ and $M_{\kappa_{e}}C_{\kappa_{e}}=-C_{\kappa_{e}}M_{\kappa_{e}}$, thus 
$$\mathcal{B}(C_{\kappa_{e},0}M_{0}\jj,\overline{\jj})=\overline{\mathcal{B}(C_{\kappa_{e},0}M_{0}\jj,\overline{\jj})}+\mathrm{compact}=\Re\mathcal{B}(C_{\kappa_{e},0}M_{0}\jj,\overline{\jj})+\mathrm{compact}$$
In the same way we have 
$$\mathcal{B}(M_{0}C_{\kappa_{i},0}\jj,\overline{\jj})=\Re\mathcal{B}(M_{0}C_{\kappa_{i},0}\jj,\overline{\jj})+\mathrm{compact}.$$

Using now the Calder\'on relations \eqref{Ck2} and the fact that 
$C_{0}^*= i C_{\kappa,0}$ for $\kappa=i$, we see that
$M_{0}C_{0}^*$ is a compact perturbation of $-C_{0}^*M_{0}$.
It follows that $\big(-\frac{1}{2}\Id+M_{0}\big)^2C_{0}^*$ is a compact perturbation of $\big(-\frac{1}{2}\Id+M_{0}\big)C_{0}^*\big(-\frac{1}{2}-M_{0}\big)$ and we have
$$\mathcal{B}\left(\big(-\tfrac{1}{2}\Id+M_{0}\big)C_{0}^*\big(-\tfrac{1}{2}\Id-M_{0}\big)\jj,\overline{\jj}\right)=\mathcal{B}\left(C_{0}^*\big(-\tfrac{1}{2}\Id-M_{0}\big)\jj,\big(-\tfrac{1}{2}\Id-M_{0}\big)\overline{\jj}\right)$$
  Since $C_{0}^*$ is elliptic for the bilinear form $\mathcal{B}$ we have with Lemma~\ref{lMk}
$$
 \begin{aligned}-\mathcal{B}\left(\big(-\tfrac{1}{2}\Id+M_{0}\big)^2C_{0}^*\jj,\overline{\jj}\right)+c_{1}(\jj,\overline{\jj})&=-\mathcal{B}\left(C_{0}^*\big(\tfrac{1}{2}\Id+M_{0}\big)\jj,\big(\tfrac{1}{2}\Id+M_{0}\big)\overline{\jj}\right)\\
 &\ge \alpha_{1}||\big(\tfrac{1}{2}\Id+M_{0}\big)\jj||_{\HH_{\times}^{-\frac{1}{2}}(\Div_{\Gamma},\Gamma)}^{2}\\ 
 &\ge \alpha_{2}||\jj||_{\HH_{\times}^{-\frac{1}{2}}(\Div_{\Gamma},\Gamma)}^{2}
 - c_{2}(\jj,\overline{\jj})\,,
 \end{aligned}
$$ 
where $c_{1}(\cdot,\cdot)$ and $c_{2}(\cdot,\cdot)$ are compact bilinear forms and $\alpha_{1}$ and $\alpha_{2}$ are positive constants.

Now for brevity write $S_{0}^*=\nn\times V_{0}$ and $T_{0}^*=\Rot_{\Gamma}V_{0}\Div_{\Gamma}$.
Taking into account that $(T_{0}^*)^{2}=0$ and that 
$(S_{0}^*)^{2}:\HH_{\times}^{-\frac{1}{2}}(\Div_{\Gamma},\Gamma)\to \HH_{\times}^{-\frac{1}{2}}(\Div_{\Gamma},\Gamma)$ is compact (it maps continuously into $\HH_{\times}^{0}(\Div_{\Gamma},\Gamma)$ which is compactly imbedded in $\HH_{\times}^{-\frac{1}{2}}(\Div_{\Gamma},\Gamma)$), we get from the definitions
$$
C_{\kappa_{i},0}C_{\kappa_{e},0}C_{0}^*=-\kappa_{i}\kappa_{e}^{-1}S_{0}^*T_{0}^*S_{0}^*-\kappa_{i}^{-1}\kappa_{e}T_{0}^*S_{0}^*T_{0}^*+\mathrm{compact}.
$$
The terms $S_{0}^*T_{0}^*S_{0}^*$ and $T_{0}^*S_{0}^*T_{0}^*$ give positive contributions, namely we have
$$
\begin{aligned}
 -\mathcal{B}\big(S_{0}^*T_{0}^*S_{0}^*\jj,\overline{\jj}\big)&=\int_{\Gamma}(\rot_{\Gamma}V_{0}\overline{\jj})\cdot V_{0}(\rot_{\Gamma}V_{0}\jj)\ge
0\\
 -\mathcal{B}\big(T_{0}^*S_{0}^*T_{0}^*\jj,\overline{\jj}\big)&=\int_{\Gamma}(\Rot_{\Gamma}V_{0}\Div_{\Gamma}\overline{\jj})\cdot V_{0}(\Rot_{\Gamma}V_{0}\Div_{\Gamma}\jj)\ge
0
\end{aligned}
$$
Therefore there exists a compact bilinear form $c_{3}$ 
such that 
$$
 -\mathcal{B}(C_{\kappa_{i},0}C_{\kappa_{e},0}C_{0}^*\jj,\overline{\jj})+c_{3}(\jj,\,\overline{\jj})\ge 0\,.
$$
Collecting all the results, we can write
$$\Im\mathcal{B}\big(\SS\jj,\overline{\jj}\big)=\eta\left(-\mathcal{B}\left(C_{0}^*\big(\tfrac{1}{2}\Id+M_{0}\big)\jj,\big(\tfrac{1}{2}\Id+M_{0}\big)\overline{\jj}\right)-\mathcal{B}(C_{\kappa_{i},0}C_{\kappa_{e},0}C_{0}^*\jj,\overline{\jj})\right)-c_{S}(\jj,\,\overline{\jj})
$$ 
where
$-c_{S}(\jj,\overline{\jj})=\eta\big(c_{1}(\jj,\overline{\jj})+c_{2}(\jj,\overline{\jj})+c_{3}(\jj,\overline{\jj})\big)+\Im\mathcal{B}\big((\rho C_{\kappa_{e},0}M_{0}+M_{0}C_{\kappa_{i},0})\jj,\overline{\jj})+\mathcal{B}\big((\SS-\SS_{1})\jj,\overline{\jj})$ is a compact bilinear form and 
$$
 \Im\left(\mathcal{B}\big(\SS\jj,\overline{\jj}\big)+c_{S}(\jj,\overline{\jj})\right)\ge C_{S}||\jj||_{\HH_{\times}^{-\frac{1}{2}}(\Div_{\Gamma},\Gamma)}
$$ 
with $C_{S}=\eta\alpha_{2}$. Similar arguments can be used for the operator $\TT C_{0}^{*-1}$.
\end{proof}

\section{Integral equations 2}  \label{IE2}

 The second method is based on a layer ansatz for the interior field: We assume that the interior electric field $\EE^{i}$  can be represented either by
   $\Psi_{E_{\kappa_{i}}}\jj$ or by $\Psi_{M_{\kappa_{i}}}\jj$ where the density $\jj\in\HH^{-\frac{1}{2}}_{\times}(\Div_{\Gamma},\Gamma)$ is the unknown function we have to determine. We begin with the Stratton-Chu representation formula \eqref{(3.2)} in $\Omega^c$:
   \begin{equation}\label{ri2u2} \EE^{s}(x) = \Psi_{E^{s}}\gamma_{N_{\kappa_{e}}}^c\left(\EE^s+\EE^{inc}\right)(x)
   + \Psi_{M_{\kappa_{e}}}\gamma_{D}^c\left(\EE^s+\EE^{inc}\right)(x)\qquad x\in\Omega^c\end{equation}

We then apply the exterior traces $\gamma_{D}^c$ and $\gamma_{N_{\kappa_{e}}}^c$ and use on both sides of \eqref{ri2u2} the transmission  conditions. The result is a relation between the traces of $\EE^{i}$ on $\Gamma$:
\begin{eqnarray}
  \label{rigdu2}  
  \gamma_{D}\EE^{i}-\gamma_{D}^c\EE^{inc}&=& - \rho C_{\kappa_{e}}\gamma_{N_{\kappa_{i}}}\EE^{i} +
   \left(-\tfrac{1}{2}\Id + M_{\kappa_{e}}\right)\gamma_{D}\EE^{i},\\\label{rignu2}\hspace{.5cm}
    \rho\gamma_{N_{\kappa_{i}}}\EE^i-\gamma_{N_{\kappa_{e}}}^c\EE^{inc}&=& -C_{\kappa_{e}}\gamma_{D}\EE^{i} +
   \rho\left(-\tfrac{1}{2}\Id + M_{\kappa_{e}}\right)\gamma_{N_{\kappa_{i}}}\EE^{i}.
\end{eqnarray} 
     In the scalar case, to construct the integral equations one would simply take a linear combination of \eqref{rigdu2} and \eqref{rignu2}. Here we multiply \eqref{rigdu2} by $a$ and \eqref{rignu2} by $bC_{0}^{*}$ and subtract to obtain:
\begin{equation}
\label{eqi2} 
 \rho L'_{e}\gamma_{N_{\kappa_{i}}}\EE^{i} -
    N'_{e}\gamma_{D}\EE^{i} = h \qquad \text{ sur }\Gamma
\end{equation}
where the operators $L'_{e}$ and $N'_{e}$ are defined for all
    $\jj\in \HH_{\times}^{-\frac{1}{2}}(\Div_{\Gamma},\Gamma)$ by :
\begin{center}  $L'_{e}\jj =\left\{aC_{\kappa_{e}}-bC_{0}^{*}\left(\tfrac{1}{2}\Id+M_{\kappa_{e}}\right)\right\}\jj,$ 
\end{center} 
\begin{center} $N'_{e}\jj=\left\{-a\left(\tfrac{1}{2}\Id+M_{\kappa_{e}}\right) +bC_{0}^{*}C_{\kappa_{e}}\right\}\jj,$
\end{center}
    and
\begin{equation}  h= a\gamma_{D}\EE^{inc}
    -bC_{0}^{*}\gamma_{N_{\kappa_{e}}}\EE^{inc}\in
    \HH_{\times}^{-\frac{1}{2}}(\Div_{\Gamma},\Gamma).
\end{equation}
    
     If $\EE^{i}$ is  represented by the potential $\Psi_{E_{\kappa_{i}}}$  applied
    to a density $\jj\in \HH_{\times}^{-\frac{1}{2}}(\Div_{\Gamma},\Gamma)$    
\begin{equation}
\label{riu1}  
  \EE^{i}(x) = -(\Psi_{E_{\kappa_{i}}}\jj)(x),\qquad  x\in\Omega,
\end{equation} 
we obtain
\begin{equation}
\label{gu1}  
  \gamma_{D}\EE^{i}=C_{\kappa_{i}}\jj \qquad \text{ and }
    \qquad\gamma_{N_{\kappa_{i}}}\EE^{i}=\left(\tfrac{1}{2}\Id + M_{\kappa_{i}}\right)\jj \qquad\text{ on }\Gamma
\end{equation}
    
     Substituting \eqref{gu1} in \eqref{eqi2}, we obtain a first integral equation:
\begin{equation}
\label{St}  
  \SS'\jj \equiv \left\{\rho L'_{e}\left(\tfrac{1}{2}\Id +
    M_{\kappa_{i}}\right) -
    N'_{e}C_{\kappa_{i}}\right\}\jj = h\qquad \text{ on }\Gamma
\end{equation}
   
This is an integral equation for the unknown 
$\jj\in \HH_{\times}^{-\frac{1}{2}}(\Div_{\Gamma},\Gamma)$. 
Having solved this equation, we construct $\EE^{i}$ and $\EE^{s}$ by the representations \eqref{riu1} in $\Omega$ and
\begin{equation} 
\label{riu22} 
  \EE^{s} = \rho\left(\Psi_{E_{\kappa_{e}}}\big(\tfrac{1}{2}\Id +
    M_{\kappa_{i}}\big)\jj\right)(x) + \left(\Psi_{M_{\kappa_{e}}}C_{\kappa_{i}}\jj\right)(x) \qquad x\in\Omega^c.\end{equation}
    
      If $\EE^{i}$ is represented by the  potential
    $\Psi_{M_{\kappa_{i}}}$ applied to a density $\mm\in
    \HH_{\times}^{-\frac{1}{2}}(\Div_{\Gamma},\Gamma)$
\begin{equation}
\label{riu12}  
  \EE^{i}(x) = -(\Psi_{M_{\kappa_{i}}}\mm)(x),\qquad x\in\Omega,
\end{equation}
we obtain:
\begin{equation}
\label{gnu1}  
  \gamma_{D}\EE^{i} = \left(\tfrac{1}{2}I + M_{\kappa_{i}}\right)\mm\qquad  and
    \qquad\gamma_{N_{\kappa_{i}}}\EE^{i} = C_{\kappa_{i}}\mm\qquad\text{ on }\Gamma.
\end{equation}
    
     Substituting \eqref{gnu1} in \eqref{eqi2}, we obtain  a second integral equation:
     \begin{equation} \label{Tt} \TT'\mm \equiv \left\{\rho L'_{e}C_{\kappa_{i}} -
     N'_{e}\left(\tfrac{1}{2}\Id +
     M_{\kappa_{i}}\right)\right\}\mm = h\qquad\text{ on }\Gamma.\end{equation}
     
      This is an integral equation for the unknown $\mm\in
     \HH_{\times}^{-\frac{1}{2}}(\Div_{\Gamma},\Gamma)$. Having solved this
     equation, we construct  $\EE^{i}$ and $\EE^{s}$ by the representations \eqref{riu12} in $\Omega$ and:
\begin{equation}
\label{riu23}  
   \EE^{s}(x) = \rho\left(\Psi_{E_{\kappa_{e}}}C_{\kappa_{i}}\mm\right)(x) +
     \left(\Psi_{M_{\kappa_{e}}}\big(\tfrac{1}{2}\Id + M_{\kappa_{i}}\big)\mm\right)(x), \qquad x\in\Omega^c.
\end{equation}
     
       Contrary to the preceding method from the previous section, the two integral equations  are not equivalent, in general. The following theorem corresponds to Theorem~\ref{solpt}. The proof is similar to the scalar case.
      
\begin{theorem}
\label{solpt2} 
 We assume that  $\kappa_{e}^2$ is not an eigenvalue of the associated interior problem \eqref{intp}.\newline
     If $\jj\in \HH_{\times}^{-\frac{1}{2}}(\Div_{\Gamma},\Gamma)$ 
     solves \eqref{St}, then $\EE^{i}$ and $\EE^{s}$, given by \eqref{riu1} and \eqref{riu22}
     respectively, solve the dielectric scattering problem.\newline If
     $\mm\in \HH_{\times}^{-\frac{1}{2}}(\Div_{\Gamma},\Gamma)$ solves \eqref{Tt},
     then $\EE^{i}$ and $\EE^{s}$, given by \eqref{riu12} and \eqref{riu23} respectively, 
     solve the dielectric scattering problem.\end{theorem} 
     
\begin{proof} 
In each case the integral representations of $\EE^{i}$ and
     $\EE^{s}$ satisfy the Maxwell equations and the Silver-M\"uller  condition. It remains  to  prove that   the transmission conditions are satisfied. We  prove it for the equation \eqref{Tt}, the arguments being similar for \eqref{St}.\newline
     Assume that $\mm$ solves \eqref{Tt} which we rewrite as :
\begin{equation}
\label{eqi22}
  \begin{array}{rl}  a\left\{\rho C_{\kappa_{e}}C_{\kappa_{i}}\mm + \left(\tfrac{1}{2}\Id +
     M_{\kappa_{e}}\right)\left(\tfrac{1}{2}\Id + M_{\kappa_{i}}\right)\mm  - \gamma_{D}\EE^{inc}\right\} \hspace{.5cm}&\vspace{2mm} \\  -bC_{0}^{*}\left\{\rho\left(\tfrac{1}{2}\Id
     + M_{\kappa_{e}}\right)C_{\kappa_{i}}\mm + C_{\kappa_{e}}\left(\tfrac{1}{2}\Id + M_{\kappa_{i}}\right)\mm - \gamma_{N_{\kappa_{e}}}\EE^{inc}\right\}&
     = 0. \end{array}
\end{equation}
      Then, using the integral representation  \eqref{riu23} of $\EE^s$, we obtain :
\begin{align*} 
  (\gamma_{D}^c\EE^{s} +\gamma_{D}^c\EE^{inc}
      -\gamma_{D}\EE^{i}) &= -\rho C_{\kappa_{e}}C_{\kappa_{i}}\mm -
     \left(\tfrac{1}{2}\Id + M_{\kappa_{e}}\right)\left(\tfrac{1}{2}\Id + M_{\kappa_{i}}\right)\mm +\gamma_{D}\EE^{inc},
\\  
   (\gamma_{N_{\kappa_{e}}}^c\EE^{s} \!+\!\gamma_{N_{\kappa_{e}}}^c \!\EE^{inc}\!\!
     -\rho\gamma_{N_{\kappa_{i}}}\EE^{i})&= -\rho\left(\tfrac{1}{2}\Id \!+\!
     M_{\kappa_{e}}\right)C_{\kappa_{i}}\mm - C_{\kappa_{e}}\left(\tfrac{1}{2}\Id \!+\! M_{\kappa_{i}}\right)\mm +\gamma_{N_{\kappa_{e}}}^c\!\EE^{inc}.
\end{align*}
     
      We have to show that the right hand  sides of these equalities vanish.\newline
       We introduce the   function $\vv$ defined on $\Omega$ by 
     
     \begin{center} $\vv(x) = -\rho\Psi_{E_{\kappa_{e}}}C_{\kappa_{i}}\mm - \Psi_{M_{\kappa_{e}}}\left(\tfrac{1}{2}\Id
     + M_{\kappa_{i}}\right)\mm - \EE^{inc}$.\end{center}
     
      By equation \eqref{eqi22} we have 
$a\gamma_{D}\vv - bC_{0}^{*}\gamma_{N_{\kappa_{e}}}\vv = 0$. 
Since $\EE^{inc}$ satisfies the Maxwell system 
$\Rot\Rot\vv-\kappa_{e}^2\vv=0$ in $\Omega$, $\vv$ satisfies it, too. By  hypothesis, $\kappa_{e}^2$ is not an eigenvalue of the associated interior problem, which implies $\vv\equiv 0$ in $\Omega$. In particular, $\gamma_{D}\vv$ and
      $\gamma_{N_{\kappa_{e}}}\vv$ vanish, which shows that the above right hand sides are indeed zero and that the transmission conditions are satisfied.
\end{proof}

\begin{theorem}
Assume that the hypotheses of Theorem~\ref{t1} are satisfied and that $\kappa_{e}^2$ is not an eigenvalue of the associated interior problem \eqref{intp}.
  Then the  operators $\SS'$ and $\TT'$ are injective.
\end{theorem} 
 
\begin{proof} 
We prove the result for  the operator $\TT'$, similar  arguments being valid for $\SS'$.\newline
     Assume that $\mm_{0}\in
   \HH_{\times}^{-\frac{1}{2}}(\Div_{\Gamma},\Gamma)$ solves the homogeneous equation :
\begin{equation} 
\label{homT} 
  \TT'\mm_{0}=\rho L'_{e}C_{\kappa_{i}}\mm_{0} -
   N'_{e}\left(\tfrac{1}{2}\Id + M_{\kappa_{i}}\right)\mm_{0} = 0.
\end{equation}
    We want to show that  $\mm_{0} = 0$.\newline
    We construct  $\vv_{1}$ and $\vv_{2}$ as follows:
   \begin{center} $\vv_{2}(x) = \rho(\Psi_{E_{\kappa_{e}}}C_{\kappa_{i}}\mm_{0})(x) +
     \left(\Psi_{M_{\kappa_{e}}}\left\{\tfrac{1}{2}\Id + M_{\kappa_{i}}\right\}\mm_{0}\right)(x), \qquad x\in\Omega^c$,\end{center}
    and \begin{center} $\vv_{1}(x) = -(\Psi_{M_{\kappa_{i}}}\mm_{0})(x),\qquad x\in
    \Omega$.
\end{center} 

By Theorem \ref{solpt2}, these functions solve the homogeneous scattering problem  (i.e. when $\EE^{inc}\equiv0$), and therefore 
    $\vv_{1}\equiv 0$ in $\Omega$ and $\vv_{2}\equiv 0$ in $\Omega^c$. Now
    we define  \begin{center} $\vv(x) = -(\Psi_{M_{\kappa_{i}}}\mm_{0})(x)\qquad x\in
    \Omega^c$\end{center}
    We have $\gamma_{N_{\kappa_{i}}}^c\vv=C_{\kappa_{i}}\mm_{0}=\gamma_{N_{\kappa_{i}}}\vv_{1}=0$. Since $\vv$ satisfies the Silver-M\"uller condition, we have $\vv\equiv0$ in $\Omega^c$. Thus $\vv\equiv0$ is $\R^3$ and $[\gamma_{D}]\vv=\mm_{0}=0$. 
 \end{proof}

\begin{remark} The operators $\SS'$ and
$\TT'$ are the dual operators of  $\,\SS$ and $\TT$, respectively, for the bilinear form $\mathcal{B}$.  Therefore  they are  Fredholm of index zero under the same hypotheses as those given in Theorems \ref{thT1}, \ref{thT2} and \ref{ThLip}. 
\end{remark}

In order that each of the four integral equations   admit a unique solution for all positive real values of $\kappa_{e}$, we will now give an example of how to choose the constants $a$ and $b$ such that the associated interior problem does not admit any real eigenvalue.

We summarize all the previous results in the final theorem.
\begin{theorem}
\label{thT}
Assume that the boundary $\Gamma$ is smooth and simply connected and
\begin{itemize}
\item[(i)]  $\kappa_{e}$ is a positive real number or $\Im(\kappa_{e})>0$ and $\Re(\kappa_{e}^2)\not=0$, 
\item[(ii)] $a=1$ and $b=\left\{\begin{array}{ll}i\eta&\text{ with $\eta\in\R\backslash\{0\}$ if $\kappa^2_{e}\in\R$} \\-i\eta\kappa_{e}\cdot\mathrm{sign}(\Im(\kappa_{e}^2))&\text{ with $\eta\in\R$, $\eta>0$ otherwise,}\end{array}\right.$
\item[(iii)] $\dfrac{\mu_{i}}{\mu_{e}}\not=-1$, $\dfrac{\mu_{e}\kappa_{i}^2}{\mu_{i}\kappa_{e}^2}\not=-1$.
\end{itemize}
  Then the operators $\SS$, $\TT$, $\SS'$ and $\TT'$ are invertible in 
  $\HH_{\times}^{-\frac{1}{2}}(\Div_{\Gamma},\Gamma)$. 
 Moreover, given the electric incident field  $\EE^{inc}\in\HH_{\loc}(\Rot,\R^3)$,
 the four integral equations  
 \eqref{S}, \eqref{T}, \eqref{St}, \eqref{Tt}
 each have a unique solution, and
 the integral representations {\rm \{\eqref{u2c}, \eqref{u12}\}}, {\rm\{\eqref{riu1},  \eqref{riu22}\}} and {\rm\{\eqref{riu12}, \eqref{riu23}\}} of $\EE^{i}$ and $\EE^{s}$ give the  solution of the dielectric scattering problem.\\
If $\Gamma$ is only Lipschitz, then the conclusions remain valid if the conditions (i) to (iii) are replaced by the more restrictive assumptions
\begin{itemize}
\item[(iv)] $\mu_{e}$, $\mu_{i}$, $\kappa_{e}$ and $\kappa_{i}$ are  positive real numbers.
\item[(v)] $a=1$ and $b=-i\eta$ with $\eta\in\R$, $\eta>0$,
\end{itemize}
 
 \end{theorem}

 \section{Discussion}
 In this paper we have described and analyzed modified boundary integral equations to 
solve a radiation problem for the Maxwell system that are stable for all wave numbers. Generalizing the approach of Kleinman and Martin to the Maxwell system by employing a suitable regularizing operator introduced by Steinbach and Windisch,
in Section~\ref{IE1} we have derived two boundary integral equations using an ansatz for the exterior field and in Section~\ref{IE2} we have derived two integral equations using an ansatz for the interior field. Note that if it is only the exterior field that is of interest, one can choose  an integral equation which gives a simple representation for $\EE^{s}$, e.g. \eqref{S} or \eqref{T}. This choice was used in the PhD thesis \cite{FLL} for an application to a shape optimization problem involving the far field pattern \cite{FLL3}.  For numerical results using this method, we refer to \cite{FLL}.

In \cite{MartinOla} P. A. Martin and P. Ola established the existence and the uniqueness of the solution to an integral equation analogous to \eqref{S} for all real values of the exterior wave number by adapting a regularization method that was introduced by Kress \cite{Col} in the framework of  spaces of continuous functions, namely by using the operator  $\jj\mapsto\nn\times V_{0}^2\jj$ in the place of our $C_{0}^{*}$. This technique  would not yield four families of Fredholm boundary integral operators  of index zero in $\HH_{\times}^{-\frac{1}{2}}(\Div_{\Gamma},\Gamma)$ since  the invertibility of the regularizing operator is needed in our arguments (see the proof of Theorem \ref{thT1}). A more interesting advantage of the operator $C_{0}^*$ is the  possibility to use our regularization method on Lipschitz boundaries since this operator still is elliptic. 

Numerical analysis using similar CFIE-based methods for the scattering of homogeneous penetrable objects are presented in \cite{ValdesMichielssen}. 
The proposed integral equation also contains double and triple operator products. Stable discretization of these operator products can be obtained by multiplying matrices arising from the discretization of the various operators using specific basis and testing functions \cite{AndriulliMichielssen}. As is shown there, preconditioners from the same class of operators are also easily constructed. The conclusion is that appropriately preconditioned CFIE-based single-source formulations are more efficient than the coupled integral equations.

 \small

\end{document}